\documentclass[11pt,reqno,twoside]{amsart}
\usepackage{graphicx}
\usepackage{amssymb}
\usepackage{epstopdf}
\usepackage[asymmetric,top=3.5cm,bottom=4.3cm,left=3.1cm,right=3.1cm]{geometry}
\usepackage{booktabs} % for much better looking tables
\usepackage{array} % for better arrays (eg matrices) in maths
\usepackage{paralist} % very flexible & customisable lists (eg. enumerate/itemize, etc.)
\usepackage{verbatim} % adds environment for commenting out blocks of text & for better verbatim
\usepackage{subfig} % make it possible to include more than one captioned figure/table in a single float
\usepackage{tabularx}
\usepackage{amsmath,amsfonts,amsthm,mathrsfs,amssymb,cite}
\usepackage[usenames]{color}
\usepackage{bm}

\newtheorem{thm}{Theorem}[section]

\newtheorem{lem}{Lemma}[section]
\newtheorem{prop}{Proposition}[section]
\theoremstyle{definition}

\theoremstyle{remark}

\newtheorem{rem}{Remark}[section]
\numberwithin{equation}{section}

\newcommand{\bu}{\mathbf{u}}
\newcommand{\bv}{\mathbf{v}}

\newcommand{\bw}{\mathbf{w}}
\newcommand{\bff}{\mathbf{f}}
\newcommand{\bvarphi}{\bm{\varphi}}
\newcommand{\bpsi}{\bm{\psi}}
\newcommand{\bnu}{\bm{\nu}}
\newcommand{\bGa}{\mathbf{\Gamma}}
\newcommand{\bx}{\mathbf{x}}
\newcommand{\by}{\mathbf{y}}
\newcommand{\rmi}{\mathrm{i}}
\newcommand{\bS}{\mathbf{S}}
\newcommand{\bK}{\mathbf{K}}
\newcommand{\bI}{\mathbf{I}}
\newcommand{\bbS}{\mathbb{S}}
\newcommand{\bA}{\mathbf{A}}
\newcommand{\ba}{\mathbf{a}}
\newcommand{\bY}{\mathbf{Y}}
\newcommand{\bC}{\mathbf{C}}
\newcommand{\bc}{\mathbf{c}}
\newcommand{\be}{\mathbf{e}}

\newcommand{\bF}{\mathbf{F}}

\newcommand{\Ical}{\mathcal{I}}

\newcommand{\Lcal}{\mathcal{L}}

\newcommand{\Ncal}{\mathcal{N}}
\newcommand{\Ocal}{\mathcal{O}}
\newcommand{\Tcal}{\mathcal{T}}
\newcommand{\Ucal}{\mathcal{U}}
\newcommand{\Vcal}{\mathcal{V}}

\newcommand{\eqnref}[1]{(\ref {#1})}

\newcommand{\p}{\partial}

\newcommand{\beq}{\begin{equation}}
\newcommand{\eeq}{\end{equation}}

\title[Polariton resonances and cloaking beyond quasistatic limit in elasticity]{Spectral properties of Neumann-Poincar\'e operator and anomalous localized resonance in elasticity beyond quasi-static limit}

\author{Youjun Deng}
\address{School of Mathematics and Statistics, Central South University, Changsha, Hunan, P. R. China.}
\email{youjundeng@csu.edu.cn, dengyijun\_001@163.com}

\author{Hongjie Li}
\address{Department of Mathematics, Hong Kong Baptist University, Kowloon Tong, Hong Kong SAR.}
\email{hongjie$_-$li@yeah.net}

\author{Hongyu Liu}
\address{Department of Mathematics, Hong Kong Baptist University, Kowloon Tong, Hong Kong SAR.\vspace*{-4mm}}
\address{\vspace*{-4mm}and}
\address{HKBU Institute of Research and Continuing Education, Virtual University Park, Shenzhen, P. R. China.}
\email{hongyu.liuip@gmail.com}

\begin{document}
\maketitle

\begin{abstract}

This paper is concerned with the polariton resonances and their application for cloaking due to anomalous localized resonance (CALR) for the elastic system within finite frequency regime beyond the quasi-static approximation. We first derive the complete spectral system of the Neumann-Poincar\'e operator associated with the elastic system within the finite frequency regime. Based on the obtained spectral results, we construct a broad class of elastic configurations that can induce polariton resonances beyond the quasi-static limit. As an application, the invisibility cloaking effect is achieved through constructing a class of core-shell-matrix metamaterial structures provided the source is located inside a critical radius. Moreover, if the source is located outside the critical radius, it is proved that there is no resonance.

\medskip

\medskip

\noindent{\bf Keywords:}~~anomalous localized resonance, negative material,  core-shell structure, beyond quasistatic limit, Neumann-Poinc\'are operator, spectral


\end{abstract}

\section{Introduction}
Recently, there are considerable mathematical studies on the plasmon resonances in order to gain deep understandings about their distinctive properties and investigate potential applications. Plasmon resonances are associated to wave interactions with metamaterials, which are artificially engineered and may possess negative material parameters. In the resonant state, due to the excitation of an appropriate source, the induced field exhibits highly oscillatory behaviours in a certain peculiar manner. Mathematically, the phenomenon of plasmon resonances is connected to an infinite dimensional set of the so-called perfect plasmon waves, which are actually the kernel of a certain non-elliptic partial differential operator (PDO) arising from the underlying physical system. More specifically, the presence of the negative material parameters breaks the ellipticity of the aforementioned PDO and thus the PDO may have a nontrivial kernel space. On the other hand, through an integral reformulation via the potential-theoretic approach, the plasmon resonance can be connected to the spectral system of the so-called Neumann-Poincar\'e (N-P) operators, which are a certain type of boundary layer potential operators. Hence, in order to understand the plasmon resonances, one needs to achieve thorough understandings of the spectral properties of certain PDOs or integral operators in various scenarios that were unveiled before. Those connections make the mathematical study of plasmon resonances a fascinating topic. For related studies in the literature, we refer to \cite{ADM, AMRZ, AKL, KLO} for the acoustic wave system and \cite{Ack13, Acm13, Ack14, ARYZ, BLL, Bos10, Brl07, Klsap, LLL, LLLW, GWM1,GWM3, GWM4,GWM5, GWM6, GWM7, GWM8} for the Maxwell system.

One particularly interesting type of plasmon resonances is the anomalous localized resonance (ALR), which is also one of the focuses of the present study. The localized feature refers to the fact that the resonance is spatially localized; that is, the corresponding field only diverges in a certain region with a sharp boundary not defined by any discontinuities in the parameters and outside that region, the field converges to a smooth one. Moreover, the resonance region moves as the position of the source is moved. Indeed, ALR heavily depends on the form as well as the location of the source term. For a fixed plasmonic configuration, if the source is located inside a critical radius, then ALR occurs, whereas if it is located outside the critical radius, then resonance does not occur. One appealing feature of ALR is that it can induce the cloaking effect; that is, if ALR occurs, then both the plasmonic configuration and the source are invisible with respect to observations outside a certain region. This cloaking phenomenon is referred to as cloaking due to anomalous localized resonance (CALR). CALR was first observed and rigorously justified by Milton and Nicorovici in \cite{GWM3} and was further studied by Ammari et al in \cite{Ack13}. We refer to the papers \cite{Ack13, Klsap, GWM3, GWM5, GWM7} and references therein for more discussions. Similar resonance phenomena were observed and investigated in elasticity \cite{AJKKY15, AKKY16, DLL, DLL1, LiLiu2d, LiLiu3d, LLL1}, which are referred to as polariton resonances in the literature. In this paper, we are mainly concerned with the polariton resonances and their application for CALR for the elastic system within the finite frequency regime beyond the quasi-static approximation. In what follows, we first present the mathematical formulation for our subsequent discussion and study.

 Let $\mathbf{C}(\bx):=(\mathrm{C}_{ijkl}(\bx))_{i,j,k,l=1}^3$, $\bx\in\mathbb{R}^3$ be a four-rank elastic material tensor defined by
 \begin{equation}\label{eq:ten}
 \mathrm{C}_{ijkl}(\bx):=\lambda(\bx)\bm{\delta}_{ij}\bm{\delta}_{kl}+\mu(\bx)(\bm{\delta}_{ik}\bm{\delta}_{jl}+\bm{\delta}_{il}\bm{\delta}_{jk}),\ \ \bx\in\mathbb{R}^3,
 \end{equation}
 where $\bm{\delta}$ is the Kronecker delta. In \eqref{eq:ten}, $\lambda$ and $\mu$ are two scalar functions and referred to as the Lam\'e parameters. For a regular elastic material, the Lam\'e parameters satisfy the following two strong convexity conditions,
 \begin{equation}\label{eq:con}
  \mathrm{i)}.~~\mu>0\qquad\mbox{and}\qquad \mathrm{ii)}.~~3\lambda+2\mu>0.
 \end{equation}
Let $D, \Omega\subset\mathbb{R}^3$ with $D\subset\Omega$ be two bounded domains with connected Lipschitz boundaries. Assume that the domain $\mathbb{R}^3\backslash\overline{\Omega}$ is occupied by a regular elastic material parameterized by the Lam\'e constants $(\lambda,\mu)$ satisfying the strong convexity conditions in \eqref{eq:con}. The shell $\Omega\backslash\overline{D}$ is occupied by a metamaterial whose Lam\'e parameters are given by $(\hat{\lambda}, \hat{\mu})$,
{ where $(\hat{\lambda}, \hat{\mu})\in\mathbb{C}^2$ with $\Im \hat{\lambda}>0, \Im \hat{\mu}>0$, which shall be properly chosen in what follows.}
%{\color{red}\begin{equation}\label{eq:copla}
%\hat{\lambda}=\hat{\lambda}_1+\rmi\delta \qquad\mbox{and}\qquad \hat{\mu}=\hat{\mu}_1+\rmi\delta,\quad\delta\in\mathbb{R}_+,
%\end{equation}
%where $\rmi=\sqrt{-1}$, and $\hat{\lambda}_1, \hat{\mu}_1$ shall be properly chosen in what follows.}
%In \eqref{eq:copla}, $\hat{\lambda}_1, \hat{\mu}_1\in\mathbb{C}$ with $\Im{\hat{\lambda}_1}\geq0$, $\Im{\hat{\mu}_1}\geq0$ and $\delta>0$.
Finally, the inner core $D$ is occupied by a regular elastic material $(\breve{\lambda}, \breve{\mu})$ satisfying the strong convex conditions \eqref{eq:con}. Denote by $\mathbf{C}_{\mathbb{R}^3\backslash\overline{\Omega},\lambda,\mu}$ to specify the dependence of the elastic tensor on the domain $\mathbb{R}^3\backslash\overline{\Omega}$ and the Lam\'e parameters $(\lambda,\mu)$. The same notation also applies for the tensors $\mathbf{C}_{\Omega\backslash\overline{D},\hat{\lambda},\hat{\mu}}$ and $\mathbf{C}_{D,\breve{\lambda},\breve{\mu}}$. Now we introduce the following elastic tensor
\begin{equation}\label{eq:pa1}
 \mathbf{C}_0=\mathbf{C}_{\mathbb{R}^3\backslash\overline{\Omega},\lambda,\mu} + \mathbf{C}_{\Omega\backslash\overline{D},\hat{\lambda},\hat{\mu}} + \mathbf{C}_{D,\breve{\lambda},\breve{\mu}}.
\end{equation}
$\mathbf{C}_0$ describes an elastic material configuration of a core-shell-matrix structure with the metamaterial located in the shell.
Let $\bff\in H^{-1}(\mathbb{R}^3)^3$ signify an excitation elastic source that is compactly supported in $\mathbb{R}^3\backslash\overline{\Omega}$. The induced elastic displacement field $\bu=(u_i)_{i=1}^3\in\mathbb{C}^3$ corresponding to the configurations described above is governed by the following PDE (partial differential equation) system
\begin{equation}\label{eq:lame1}
\begin{cases}
& \nabla\cdot\mathbf{C}_0\nabla^s\mathbf{u}(\mathbf{x})+\omega^2 \mathbf{u}(\bx)=\mathbf{f}\quad\mbox{in}\ \ \mathbb{R}^3,\medskip\\
& \mbox{$\mathbf{u}(\bx)$ satisfies the radiation condition,}
\end{cases}
\end{equation}
where $\omega\in\mathbb{R}_+$ is the angular frequency, and the operator $\nabla^s$ is the symmetric gradient given by
 \begin{equation}\label{eq:sg1}
 \nabla^s\mathbf{u}:=\frac{1}{2}\left(\nabla\mathbf{u}+\nabla\mathbf{u}^t \right),
 \end{equation}
 with $\nabla\bu$ denoting the matrix $(\partial_j u_i)_{i,j=1}^3$ and the superscript $t$ signifying the matrix transpose. In \eqref{eq:lame1}, the radiation condition designates the following condition as $|\mathbf{x}|\rightarrow+\infty$ (cf. \cite{Kup}),
\begin{equation}\label{eq:radi}
\begin{split}
(\nabla\times\nabla\times\mathbf{u})(\bx)\times\frac{\bx}{|\bx|}-\mathrm{i}k_s\nabla\times\mathbf{u}(\bx)=&\mathcal{O}(|\bx|^{-2}),\\
\frac{\bx}{|\bx|}\cdot[\nabla(\nabla\cdot\mathbf{u})](\bx)-\mathrm{i}k_p\nabla\mathbf{u}(\bx)=&\mathcal{O}(|\bx|^{-2}),
\end{split}
\end{equation}
where $\rmi=\sqrt{-1}$ and
\begin{equation}\label{pa:ksp}
 k_s=\omega/\sqrt{\mu}, \quad  k_p=\omega/\sqrt{\lambda+2\mu},
\end{equation}
with $\lambda$ and $\mu$ defined in \eqref{eq:pa1}.

Next we introduce the following functional for $\bw,\bv\in \big(H^1(\Omega\backslash\overline{D})\big)^3$,
\begin{equation}\label{eq:func1}
\begin{split}
 P_{\hat{\lambda},\hat{\mu}}(\bw,\bv) = &\int_{\Omega\backslash\overline{D}} \nabla^s\bw:\mathbf{C}_0\overline{\nabla^s\bv(\bx)}d\bx \\
 = & \int_{\Omega\backslash \overline{D}}\Big(\hat{\lambda}(\nabla\cdot\mathbf{\bw})\overline{(\nabla\cdot\mathbf{\bv})}(\bx)+2\hat{\mu}\nabla^s\mathbf{w}:\overline{\nabla^s\mathbf{\bv}}(\bx) \Big)\ d \bx,
 \end{split}
\end{equation}
where  $\mathbf{C}_0$ and $\nabla^s$ are defined in \eqref{eq:pa1} and \eqref{eq:sg1}, respectively. In \eqref{eq:func1} and also in what follows, $\mathbf{A}:\mathbf{B}=\sum_{i,j=1}^3 a_{ij}b_{ij}$ for two matrices $\mathbf{A}=(a_{ij})_{i,j=1}^3$ and $\mathbf{B}=(b_{ij})_{i,j=1}^3$. Henceforth, we define
\begin{equation}\label{def:E}
E(\bu)=\Im P_{\hat{\lambda},\hat{\mu}}(\bu,\bu),
\end{equation}
which signifies the energy dissipation exists energy of the elastic system \eqref{eq:lame1}.
We are now in a position to present the definition of CALR. We say that polariton resonance occurs if for any $M\in\mathbb{R}_+$,
\begin{equation}\label{con:res}
 E(\bu) \geq M,
\end{equation}
where $\bu$ depends on the Lam\'e parameters $(\hat{\lambda}, \hat{\mu})$. In addition to \eqref{con:res}, if the displacement field $\bu$ further satisfies the following boundedness condition,
\begin{equation}\label{con:bou}
|\bu|\leq C,\quad \mbox{when} \quad |\bx|>\tilde{R},
\end{equation}
for a certain $\tilde{R}\in\mathbb{R}_+$, which does not depend on the Lam\'e parameters $(\hat{\lambda}, \hat{\mu})$, then we say that CALR occurs. We refer to \cite{Ack13} and \cite{GWM3} for more relevant discussions.

In this paper, we aim to construct a broad class of elastic structures that can induce polariton resonances and CALR. It is emphasized that we shall not require the following quasi-static condition throughout our study,
\begin{equation}\label{eq:qs1}
\omega\cdot\mathrm{diam}(\Omega)\ll 1.
\end{equation}
The quasi-static approximation \eqref{eq:qs1} has played a critical role in all of the existing studies concerning the polariton resonances for the elastic system \cite{AJKKY15, AKKY16, DLL, LiLiu2d, LiLiu3d, LLL1} as mentioned before. In fact, \cite{AJKKY15, AKKY16, DLL, LiLiu2d, LiLiu3d} consider the static case by directly taking $\omega\equiv 0$ and \cite{LLL1} rigorously verifies the quasi-static approximation. One of the major contributions of this work is the construction of a class of core-shell-matrix polariton structures that can induce CALR within the finite frequency beyond the quasi-static approximation in elasticity. Moreover, our construction of the material structures is very broad in the following sense. In \cite{AJKKY15, AKKY16, DLL, LiLiu2d, LiLiu3d}, the metamaterial parameters were constructed such that both the two strong convexity conditions in \eqref{eq:con} are violated. In our study, the metamaterial parameters are constructed such that only one of the two strong convexity conditions is required to be violated. It is noted that in \cite{LLL1}, the resonant construction also only requires the violation of any one of the two convexity conditions. However, the study in \cite{LLL1} is mainly concerned with the static case. Indeed, we show that the CALR construction in the current work includes the constructions in \cite{DLL, LLL1} as special cases by taking the quasi-static limit. Finally, in order to establish the aforementioned results, we make essential use of spectral arguments. We derive the complete spectral system of the N-P operator associated to the elastic system with the finite-frequency regime. It is remarked that the corresponding derivation is highly nontrivial and the spectral results are of significant mathematical interest for their own sake.

The main results of this paper can be sketched as follows. In Theorem~\ref{thm:ks}, we derive the complete spectral system of the N-P operator within spherical geometry and finite-frequency regime. It is remarked that that in the static case, the spectral system of the N-P operator was derived in \cite{DLL}. We show that by taking the quasi-static limit in our spectral result obtained in Theorem~\ref{thm:ks}, one can actually derives the result in \cite{DLL}; see Remark~\ref{rem:sta}. That is, the spectral result in Theorem~\ref{thm:ks} generalizes and extends the result in \cite{DLL} beyond the quasi-static limit. In Theorem~\ref{thm:reson}, by taking $D=\emptyset$ and $\Omega=B_R$ with $B_{R}$ a central ball of radius $R$, we show that the polariton resonance occurs for a broad class of sources $\bf$ provided the Lam\'e parameter $\hat{\mu}$ inside the domain $\Omega$ satisfies the condition \eqref{con:001}. In Theorem~\ref{thm:CALR}, by letting $D=B_{r_i}$ and $\Omega=B_{r_e}$, the Newtonian potential $\bF$ of the source term $\bff$ be given in \eqref{eq:FF2}, and the Lam\'e parameters $\breve{\mu}$ and $\hat{\mu}$ satisfy the condition \eqref{eq:FF2}, we show that CALR occurs provided the source $\bf$ is supported inside a critical radius $r_*=\sqrt{r_e^3/r_i}$. We also show that if the source is located outside the critical radius, then no resonance occurs.

Three remarks are in order. First, it is noted that we mainly work within the spherical geometry. Indeed, we shall require the exact spectral information of the N-P operator. Beyond the spherical geometry, it is rather unpractical to derive the required spectral results. In fact, even in the simplest electro-static case, only the radical geometry \cite{Ack13} and ellipse geometry \cite{AK} were considered. For more general geometries, one may resort to the assistance of numerical simulations; see \cite{BL} for the electro-static case. Second, when deriving the polariton resonance and the CALR, we only need to have constraints on the Lam\'e parameter $\hat{\mu}$ and require no restriction on the other parameter $\hat{\lambda}$, which makes our theoretical constructions easier for applications. Third, in Theorem \ref{thm:CALR} on CALR, the Newtonian potential $\bF$ of the source term $\bff$ is assumed to have the expression in \eqref{eq:FF2}.  This constraint on $\bff$ is only a technical issue. In fact, the ALR is a spectral phenomenon at the accumulating point of the eigenvalues of the N-P operator, which naturally requires that the order $n_0$ in Theorem \ref{thm:CALR} should be large; see Remark \ref{rem:genf} for more relevant discussions.

The rest of the paper is organized as follows. Section 2 is devoted to the preliminaries on some notations and layer potentials of the elastic system. In Section 3, the complete spectral system of the N-P operator is derived. Sections 4 and 5 are respectively devoted to the polariton resonance and CALR results.

\section{Preliminaries}
In this section, we present some preliminary knowledge for the elastic system for our subsequent use. We first introduce the elastostatic operator $\Lcal_{\lambda,\mu}$ associated to the Lam\'e constants $(\lambda, \mu)$ as follows,
\begin{equation}\label{op:lame}
 \Lcal_{\lambda,\mu}\bw:=\mu \triangle\bw + (\lambda+ \mu)\nabla\nabla\cdot\bw,
\end{equation}
for $\bw\in\mathbb{C}^3$.
The traction (the conormal derivative) of $\bw$ on $\partial \Omega$ is defined to be
\begin{equation}\label{eq:trac}
\partial_{\bnu}\bw=\lambda(\nabla\cdot \bw)\bnu + 2\mu(\nabla^s\bw) \bnu,
\end{equation}
where $\nabla^s$ is defined in \eqref{eq:sg1} and $\bnu$ is the outward unit normal to the boundary $\partial \Omega$.

From \cite{Kup}, the fundamental solution $\bGa^{\omega}=(\Gamma^{\omega}_{i,j})_{i,j=1}^3$ for the operator $\Lcal_{\lambda,\mu}+\omega^2$ in three dimensions is given by
\begin{equation}\label{eq:ef}
 (\Gamma^{\omega}_{i,j})_{i,j=1}^3(\bx)=-\frac{\bm{\delta}_{ij}}{4\pi\mu|\bx|}e^{\rmi k_s |\bx|} + \frac{1}{4\pi \omega^2}\partial_i\partial_j\frac{e^{\rmi k_p|\bx|} - e^{\rmi k_s|\bx|}}{|\bx|},
\end{equation}
where $k_s$ and $k_p$ are defined in \eqref{pa:ksp}. Then the single layer potential associated with the fundamental solution $\bGa^{\omega}$ is defined as
\begin{equation}\label{eq:single}
 \bS_{\partial\Omega}^{\omega}[\bvarphi](\bx)=\int_{\partial \Omega} \bGa^{\omega}(\bx-\by)\bvarphi(\by)ds(\by), \quad \bx\in\mathbb{R}^3,
\end{equation}
for $\bvarphi\in L^2(\partial \Omega)^3$. On the boundary $\partial \Omega$, the conormal derivative of the single layer potential satisfies the following jump formula
\begin{equation}\label{eq:jump}
 \frac{\partial \bS_{\partial\Omega}^{\omega}[\bvarphi]}{\partial \bnu}|_{\pm}(\bx)=\left( \pm\frac{1}{2}\bI + \left(\bK_{\partial\Omega}^{\omega}\right)^*  \right)[\bvarphi](\bx) \quad \bx\in\partial \Omega,
\end{equation}
where
\[
 (\bK_{\partial\Omega}^{\omega})^*[\bvarphi](\bx)=\mbox{p.v.} \int_{\partial \Omega} \frac{\partial \bGa^{\omega}}{\partial \bnu(\bx)}(\bx-\by)\bvarphi(\by)ds(\by),
\]
with $\mbox{p.v.}$ standing for the Cauchy principal value and the subscript $\pm$ indicating the limits from outside and inside $\Omega$, respectively. The operator $(\bK_{\partial\Omega}^{\omega})^*$ is called to be the Neumann-Poincar\'e (N-P) operator.

Let $ \Phi(\bx)$ be the fundamental solution to the operator $\triangle + \omega^2$ in three dimensions given as follows
\begin{equation}\label{eq:ful}
  \Phi(\bx)=-\frac{e^{\rmi \omega\bx}}{4\pi|\bx|}.
\end{equation}
For $\varphi\in L^2(\partial\Omega)$, we define
\begin{equation}\label{eq:sh}
 S_{\partial\Omega}^{\omega}[\varphi](\bx)=\int_{\partial \Omega} \Phi(\bx-\by)\varphi(\by)ds(\by), \quad \bx\in\mathbb{R}^3.
\end{equation}

Next, to facilitate the exposition, we present some notations and useful formulas. Let $\mathbb{N}$ be the set of the positive integers and $\mathbb{N}_0=\mathbb{N}\cup \{0\}$. Set $Y_n^m$ with $n\in\mathbb{N}_0, -n\leq m \leq n$ to be the spherical harmonic functions. Let $\bbS_{R}$ be the surface of the ball $B_R$ and denote by $\bbS$ for $R=1$ for simplicity. Furthermore, the operators $\nabla_{\bbS}$, $\nabla_{\bbS}\cdot$ and $\triangle_{\bbS}$ designate the surface gradient, the surface divergence and the Laplace-Beltrami operator on the unit sphere $\bbS$.

Let $j_n(t)$ and $h_n(t)$, $n\in\mathbb{N}_0$, denote the spherical Bessel and Hankel functions of the first kind of order $n$, respectively.
The following asymptotic expansions shall be needed in what follows (cf. \cite{CK}),
\begin{equation}\label{eq:asj}
\begin{split}
  j_n(t)=  \frac{t^n}{(2n+1)!!}\left(1+ \grave{j}(t)\right), \quad
  h_n(t)=  \frac{ (2n-1)!!}{\rmi t^{n+1}}\left(1+ \grave{h}(t) \right),
 \end{split}
\end{equation}
for $n\gg 1$, where
\[
  \grave{j}(t) = \mathcal{O}\left(\frac{1}{n}\right)  \quad \mbox{and} \quad  \grave{h}(t) = \mathcal{O}\left(\frac{1}{n}\right);
\]
and for a fixed $n$ with $t\ll 1$,
\begin{equation}\label{eq:asjt}
\begin{split}
  j_n(t)=  \frac{t^n}{(2n+1)!!}\big(1+ \Ocal(t)\big), \quad
  h_n(t)=  \frac{ (2n-1)!!}{\rmi t^{n+1}}\big(1+ \Ocal(t) \big).
 \end{split}
\end{equation}

The following three auxiliary lemmas shall be needed as well \cite{Jcn}.

\begin{lem}\label{lem:1}
For a vector field { $\bw\in H^1(\bbS)^3$} and a scalar function $v \in H^1(\bbS)$,  there hold the following relations
\begin{equation}
\begin{split}
 &\nabla_{\bbS} \cdot (\nabla_{\bbS}v \wedge \bnu) =0,\quad \triangle_{\bbS} v =\nabla_{\bbS} \cdot \nabla_{\bbS}v, \\
 & \int_{\bbS} \nabla_{\bbS} v \cdot  \bw ds =-\int_{\bbS} v \nabla_{\bbS}\cdot \bw ds,
\end{split}
\end{equation}
and
\[
 \nabla_{\bbS} \cdot (\bw v)= \nabla_{\bbS} \cdot \bw v+ \bw \cdot \nabla_{\bbS}v.
\]
\end{lem}

\begin{lem}\label{lem:2}
The spherical harmonic functions $Y_n^m$ with $n\in\mathbb{N}_0, -n\leq m \leq n$, are the eigenfunctions of the Laplace-Beltrami operator $\triangle_{\bbS}$ associated with the eigenvalue $-n(n+1)$, namely
\[
 \triangle_{\bbS} Y_n^m+n(n+1)Y_n^m=0.
\]
\end{lem}

\begin{lem}\label{lem:vec}
The family $(\Ical_n^m, \Tcal_n^m, \Ncal_n^m)$, the vectorial spherical harmonics of order $n$,
\[
\begin{split}
 \Ical_n^m= &\nabla_{\bbS}Y_{n+1}^m +(n+1) Y_{n+1}^m \bnu, \quad n\geq0, \; n+1\geq m \geq -(n+1),\\
 \Tcal_n^m= &\nabla_{\bbS}Y_{n}^m\wedge\bnu, \qquad\qquad\qquad  n\geq1, \; n\geq m \geq -n,\\
 \Ncal_n^m=&-\nabla_{\bbS}Y_{n-1}^m + nY_{n-1}^m \bnu, \quad n\geq1, \; n+1\geq m \geq -(n+1),
\end{split}
\]
forms an orthogonal basis of $(L^2(\bbS))^3$.
\end{lem}
%\begin{rem}
%{ \color{red} \bf When $n\leq 1$, just set them to be $0$. }
%\end{rem}

From Lemma \ref{lem:vec}, one has that
\[
  \Ical_{n-1}^m= \nabla_{\bbS}Y_{n}^m +n Y_{n}^m \bnu,
\]
which is a vectorial spherical harmonics of order $n-1$. Thus $\Ical_{n-1}^m$ can be expressed by
\begin{equation}\label{eq:coa}
 \Ical_{n-1}^m=\bA_{n-1,m}\bY_{n-1},
\end{equation}
where
\[
 \bY_{n-1}=[Y_{n-1}^{-(n-1)},\cdots, Y_{n-1}^{n-1}]^T,
\]
and $\bA_{n-1,m}$ is a $3\times(2n-1)$ matrix given by
\[
  \bA_{n-1,m}=[\ba_{n-1,m}^{-(n-1)}, \cdots, \ba_{n-1,m}^{n-1}].
\]
Similarly, the vectorial spherical harmonics $\Ncal_{n+1}^m$ is of order $n+1$. Hence, it can be expressed as
\begin{equation}\label{eq:coc}
 \Ncal_{n+1}^m=\bC_{n+1,m}\bY_{n+1},
\end{equation}
where $\bC_{n+1,m}$ is a $3\times(2n+3)$ matrix given as follows
\[
  \bC_{n+1,m}=[\bc_{n+1,m}^{-(n+1)}, \cdots, \bc_{n+1,m}^{n+1}].
\]

Next, we prove three important propositions.
\begin{prop}\label{pr1}
The following identities hold
\[
\begin{split}
 \int_{\bbS} \overline{Y}_{n-1}^qY_n^m\bnu ds=\frac{\ba_{n-1,m}^q}{2n+1}, \quad  &\int_{\bbS} \overline{Y}_{n-1}^q \nabla_{\bbS}Y_n^m ds=\frac{n+1}{2n+1} \ba_{n-1,m}^q,\\
  \int_{\bbS} \overline{Y}_{n+1}^qY_n^m\bnu ds=\frac{\bc_{n+1,m}^q}{2n+1}, \quad  &\int_{\bbS} \overline{Y}_{n+1}^q \nabla_{\bbS}Y_n^m ds=\frac{-n}{2n+1} \bc_{n+1,m}^q,
 \end{split}
\]
and
\[
 \int_{\bbS} \overline{Y}_{p}^qY_n^m\bnu ds=0, \quad  \int_{\bbS} \overline{Y}_{p}^q \nabla_{\bbS}Y_n^m ds=0, \quad \mbox{for} \quad p\geq0, \; p\neq n-1, n+1,
\]
where and also in what follows, the ovelrine denotes the complex conjugate. Moreover, the coefficient vectors $\ba_{n,m}^{q}$ and $\bc_{n+1,q}^{m}$, defined in \eqref{eq:coa} and \eqref{eq:coc}, satisfy the following identity
\begin{equation}\label{co:ac}
 \ba_{n,m}^{q} = \frac{2n+3}{2n+1}\overline{\bc}_{n+1,q}^{m}.
\end{equation}
\end{prop}

\begin{proof}
 From Lemma \ref{lem:vec} and the identities in \eqref{eq:coa} and \eqref{eq:coc}, one has that
 \begin{equation}\label{eq:pr1_1}
 \begin{split}
  \nabla_{\bbS}Y_{n}^m +n Y_{n}^m \bnu &= \bA_{n-1,m}\bY_{n-1}, \\
   -\nabla_{\bbS}Y_{n}^m +(n+1) Y_{n}^m \bnu &= \bC_{n+1,m}\bY_{n+1}.
 \end{split}
 \end{equation}
 Multiplying $ \overline{Y}_{n-1}^q$ on both sides of \eqref{eq:pr1_1} and integrating on the unit sphere $\bbS$ yield that
 \begin{equation}\label{eq:pr1_2}
 \int_{\bbS} \overline{Y}_{n-1}^q \nabla_{\bbS}Y_n^m ds +  n\int_{\bbS} \overline{Y}_{n-1}^qY_n^m\bnu ds =\ba_{n-1,m}^q,
 \end{equation}
 and
 \begin{equation}\label{eq:pr1_3}
-\int_{\bbS} \overline{Y}_{n-1}^q \nabla_{\bbS}Y_n^m ds +  (n+1)\int_{\bbS} \overline{Y}_{n-1}^qY_n^m\bnu ds=0.
\end{equation}
Solving the equations \eqref{eq:pr1_2} and \eqref{eq:pr1_3}, one can obtain that
\begin{equation}\label{eq:ad1}
  \int_{\bbS} \overline{Y}_{n-1}^qY_n^m\bnu ds=\frac{\ba_{n-1,m}^q}{2n+1},\ \  \int_{\bbS} \overline{Y}_{n-1}^q \nabla_{\bbS}Y_n^m ds=\frac{n+1}{2n+1} \ba_{n-1,m}^q,
\end{equation}
which are the first two identities in the proposition. By a similar argument, the other four integral identities can be proved.

The rest of the proof is to show the coefficient identity \eqref{co:ac}. Taking the complex conjugate on both sides of the equation \eqref{eq:ad1} and replacing $n$ with $n+1$ yield that
\begin{equation}\label{eq:au1}
  \int_{\bbS} \overline{Y}_{n+1}^mY_{n}^q\bnu ds=\frac{\overline{\ba}_{n,m}^q}{2n+3}.
\end{equation}
Comparing the equation \eqref{eq:au1} with the third integral identity of this proposition shows that
\[
\ba_{n,m}^{q} = \frac{2n+3}{2n+1}\overline{\bc}_{n+1,q}^{m},
\]
and this completes the proof.
\end{proof}

\begin{prop}\label{pr2}
The following identities hold
\[
\begin{split}
  &\int_{\bbS} (\nabla_{\bbS} \overline{Y}_{n-1}^q \cdot \nabla_{\bbS}Y_n^m) \bnu ds=\frac{(n+1)(n-1)}{2n+1} \ba_{n-1,m}^q,\\
  &\int_{\bbS} (\nabla_{\bbS} \overline{Y}_{n+1}^q \cdot \nabla_{\bbS}Y_n^m) \bnu ds=\frac{n(n+2)}{2n+1} \bc_{n+1,m}^q,
 \end{split}
\]
and
\[
 \int_{\bbS} (\nabla_{\bbS} \overline{Y}_{p}^q \cdot \nabla_{\bbS}Y_n^m) \bnu ds=0, \quad \mbox{for} \quad p\geq0, \; p\neq n-1, n+1,
\]
where the the coefficient vectors $\ba_{n,m}^{q}$ and $\bc_{n,m}^{q}$ are defined in \eqref{eq:coa} and \eqref{eq:coc}, respectively.
\end{prop}
\begin{proof}
From Lemmas \ref{lem:1} and \ref{lem:2}, one has by direct calculations that
\[
 \begin{split}
  &\int_{\bbS} (\nabla_{\bbS} \overline{Y}_{p}^q \cdot \nabla_{\bbS}Y_n^m) \bnu ds =\sum_{i=1}^3\be_i  \int_{\bbS} \nabla_{\bbS} \overline{Y}_{p}^q \cdot \nabla_{\bbS}Y_n^m (\bnu\cdot \be_i) ds \\
=& -\sum_{i=1}^3\be_i  \int_{\bbS} \overline{Y}_{p}^q  \nabla_{\bbS} \cdot \left( \nabla_{\bbS}Y_n^m (\bnu\cdot \be_i) \right) ds \\
=& -\sum_{i=1}^3\be_i  \int_{\bbS} \overline{Y}_{p}^q \left(  \triangle_{\bbS}Y_n^m (\bnu\cdot \be_i) +  \nabla_{\bbS}Y_n^m\cdot  \nabla_{\bbS}(\bnu\cdot \be_i)   \right)ds\\
=&  n(n+1) \int_{\bbS} \overline{Y}_{p}^q Y_n^m \bnu - \int_{\bbS} \overline{Y}_{p}^q \nabla_{\bbS}Y_n^m ds,
 \end{split}
\]
where and also in what follows, $\be_i$, $i=1,2,3$ are Euclidean unit vectors. With the help of Proposition \ref{pr1}, one can then obtain the integral identities of this proposition.

 The proof is complete.
\end{proof}

\begin{prop}\label{pr3}
The following identities hold
\[
\begin{split}
  &\int_{\bbS} \nabla_{\bbS} (\nabla_{\bbS} \overline{Y}_{n-1}^q) \cdot \nabla_{\bbS}Y_n^m ds=\frac{-n(n+1)(n-1)}{2n+1} \ba_{n-1,m}^q,\\
  &\int_{\bbS} \nabla_{\bbS} (\nabla_{\bbS} \overline{Y}_{n+1}^q) \cdot \nabla_{\bbS}Y_n^m ds=\frac{n(n+1)(n+2)}{2n+1} \bc_{n+1,m}^q,
 \end{split}
\]
and
\[
\int_{\bbS} \nabla_{\bbS} (\nabla_{\bbS} \overline{Y}_{p}^q) \cdot \nabla_{\bbS}Y_n^m ds=0, \quad \mbox{for} \quad p\geq0, \; p\neq n-1, n+1,
\]
where the the coefficient vectors $\ba_{n,m}^{q}$ and $\bc_{n,m}^{q}$ are defined in \eqref{eq:coa} and \eqref{eq:coc}, respectively.
\end{prop}
\begin{proof}
From Lemmas \ref{lem:1} and \ref{lem:2}, one has that
\[
 \begin{split}
  & \int_{\bbS} \nabla_{\bbS} (\nabla_{\bbS} \overline{Y}_{p}^q) \cdot \nabla_{\bbS}Y_n^m ds =\sum_{i=1}^3\be_i   \int_{\bbS} \nabla_{\bbS} (\nabla_{\bbS} \overline{Y}_{p}^q\cdot \be_i) \cdot \nabla_{\bbS}Y_n^m ds \\
=&  n(n+1) \int_{\bbS}\nabla_{\bbS} \overline{Y}_{p}^q Y_n^m ds .
 \end{split}
\]
 Thus the  integral identities in this proposition directly follow from Proposition \ref{pr1}.

 The proof is complete.
\end{proof}

\section{Spectral results of the Neumann-Poincar\'e operator}

In this section, we derive the complete spectral system of the N-P operator for the elastic system within the finite-frequency regime. To that end, we first derive the spectral system of the single-layer potential and then utilize the jump formulation \eqref{eq:jump} to obtain the spectral system of the  N-P operator.

From the expression of the fundamental solution $\bGa^{\omega}$ in \eqref{eq:ef}, one can readily see that
\begin{equation}\label{eq:detwo}
 \bGa^{\omega}=\bGa^{\omega}_1 + \bGa^{\omega}_2,
\end{equation}
where
\[
\bGa^{\omega}_1 = -\frac{\bm{\delta}_{ij}}{4\pi\mu|\bx|}e^{\rmi k_s |\bx|} \quad \mbox{and} \quad \bGa^{\omega}_2= \frac{1}{4\pi \omega^2}\partial_i\partial_j\frac{e^{\rmi k_p|\bx|} - e^{\rmi k_s|\bx|}}{|\bx|}.
\]
For the first part, one has  $\bGa^{\omega}_1=\Phi(\bx) \bm{\delta}_{ij}/\mu$, where $\Phi(\bx)$ is the fundamental solution of the operator $\triangle + \omega^2$ defined in \eqref{eq:ful}. Moreover, the spectral system of the operator $S_{\bbS_R}^{k}$ defined in \eqref{eq:sh}, associated with the kernel function $\Phi(\bx)$, has been derived in \cite{s25}. For the convenience of readers, we include it in the following lemma.
\begin{lem}\label{lem:sk}
The eigen-system of the single layer potential operator $S_{\bbS_R}^{k}$ defined in \eqref{eq:sh} is given as follows
\begin{equation}\label{eq:sse}
 S_{\bbS_R}^{k}[Y_n^m](\bx)=-\rmi k R^2 j_n(kR) h_n(kR) Y_n^m, \quad \bx\in\bbS_R.
\end{equation}
Moreover, the following two indentities hold
\[
 S_{\bbS_R}^{k}[Y_n^m](\bx)=-\rmi k R^2 j_n(k|\bx|) h_n(kR)Y_n^m \quad \bx\in B_R,
\]
and
\[
 S_{\bbS_R}^{k}[Y_n^m](\bx)=-\rmi k R^2 j_n(kR) h_n(k|\bx|)Y_n^m \quad \bx\in \mathbb{R}^3\backslash {B_R}.
\]
\end{lem}

Thus, we mainly focus on handling the second term $\bGa^{\omega}_2$ given in \eqref{eq:detwo}. It is noted that the fundamental solution $\Phi(\bx-\by)$ defined in \eqref{eq:ful} has the following expansion (cf. \cite{CK})
\[
 \Phi(\bx-\by) =-\rmi k \sum_{n=0}^{\infty} \sum_{m=-n}^{n} h_n(k|\bx|)Y_n^m(\hat{\bx})j_n(k|\by|)\overline{Y}_n^m(\hat{\by}) \quad \mbox{for} \quad |\by|<|\bx|.
\]
By direct calculations, there holds that
\beq\label{eq:fistdirG01}
\begin{split}
&\nabla_\by\Phi(\bx-\by) =-\rmi k \sum_{n=0}^{\infty} \sum_{m=-n}^{n} h_n(k|\bx|)Y_n^m(\hat{\bx})\nabla_\by\big(j_n(k|\by|)\overline{Y}_n^m(\hat{\by})\big)  \\
=&-\rmi k \sum_{n=0}^{\infty} \sum_{m=-n}^{n} h_n(k|\bx|)Y_n^m(\hat{\bx})\Big(j_n'(k|\by|)k\overline{Y}_n^m(\hat{\by})\hat\by+j_n(k|\by|)\nabla_{\bbS}\overline{Y}_n^m(\hat{\by})/|\by|\Big),
\end{split}
\eeq
and
\beq\label{eq:normdirG01}
\begin{split}
\frac{\p}{\bnu_\by}\nabla_\by\Phi(\bx-\by) &= -\rmi k \sum_{n=0}^{\infty} \sum_{m=-n}^{n} h_n(k|\bx|)Y_n^m(\hat{\bx}) \left(j_n''(k|\by|)k^2\overline{Y}_n^m(\hat{\by})\hat\by \right. \\
& \left. +j_n'(k|\by|)k\nabla_{\bbS}\overline{Y}_n^m(\hat{\by})/|\by|  - j_n(k|\by|)\nabla_{\bbS}\overline{Y}_n^m(\hat{\by})\hat\by /|\by|^2 \right),
\end{split}
\eeq
where
\[
 \frac{\p}{\bnu_\by}\nabla_\by\Phi(\bx-\by) =\bnu_\by\cdot \nabla_\by^2\Phi(\bx-\by).
\]

%\begin{equation}
%\begin{split}
%& \partial_i\partial_j \Phi(\bx-\by) =\rmi k \sum_{n=0}^{\infty} \sum_{m=-n}^{n} h_n(k|\bx|)Y_n^m(\hat{\bx})\times  \\
%&  \left( j_n^{\prime\prime}(k|\by|)k^2\frac{\by_i}{|\by|} \frac{\by_j}{|\by|} \overline{Y}_n^m(\hat{\by}) + j_n^{\prime}(k|\by|)k \left( \delta_{ij} \frac{1}{|\by|} - \frac{\by_i\by_j}{|\by|^3}\right) \overline{Y}_n^m(\hat{\by})  +    \right. \\
%&  \left.    j_n^{\prime}(k|\by|)k \frac{\by_i}{|\by|} \partial_j \overline{Y}_n^m(\hat{\by}) + j_n^{\prime}(k|\by|)k \frac{\by_j}{|\by|} \partial_i \overline{Y}_n^m(\hat{\by}) + j_n(k|\by|) \partial_j \partial_i \overline{Y}_n^m(\hat{\by}) \right)
%\end{split}
%\end{equation}

With the help of Propositions \ref{pr1} and \ref{pr2}, one can derive the following important result.
\begin{prop}\label{prusef01}
There hold the following identities
\beq\label{eq:imptidt01}
\begin{split}
 &\int_{\bbS_R} \nabla_\bx^2\Phi(\bx-\by)\cdot \nabla_{\bbS}Y_n^m(\hat\by) ds = \\
 & -\rmi k  h_{n-1}(k|\bx|) \left(  j_{n-1}^{\prime}(kR)kR\frac{n(n+1)}{2n+1} - j_{n-1}(kR) \frac{n(n+1)(n-1)}{2n+1} \right)  \Ical_{n-1}^m \\
 &  -\rmi k  h_{n+1}(k|\bx|) \left( j_{n+1}^{\prime}(kR)kR \frac{n(n+1)}{2n+1}   +  j_{n+1}(kR) \frac{n(n+1)(n+2)}{2n+1}  \right) \Ncal_{n+1}^m .
\end{split}
\eeq
and
\beq\label{eq:imptidt02}
\begin{split}
 &\int_{\bbS_R} \nabla_\bx^2\Phi(\bx-\by)\cdot(Y_n^m(\hat\by)\bnu_{\by}) ds = \\
  & -\rmi k  h_{n-1}(k|\bx|) \left( \left(j_{n-1}(kR) -j_{n-1}^{\prime}(kR)kR \right) \frac{n-1}{2n+1} +  j_{n-1}^{\prime\prime}(kR) k^2R^2 \frac{1}{2n+1} \right) \Ical_{n-1}^m  \\
 &  -\rmi k  h_{n+1}(k|\bx|) \left( \left(j_{n+1}^{\prime}(kR)kR-j_{n+1}(kR) \right) \frac{n+2}{2n+1}  +  j_{n+1}^{\prime\prime}(kR) k^2R^2 \frac{1}{2n+1}   \right) \Ncal_{n+1}^m .
\end{split}
\eeq
\end{prop}
\begin{proof}
Note that $\nabla_\bx^2\Phi(\bx-\by)=\nabla_\by^2\Phi(\bx-\by)$. Using integration by parts as well as Lemma \ref{lem:2}, there holds
\[
\begin{split}
&\int_{\bbS_R} \nabla_\bx^2\Phi(\bx-\by)\cdot \nabla_{\bbS}Y_n^m(\hat\by) ds = \int_{\bbS_R} \nabla_\by^2\Phi(\bx-\by)\cdot \nabla_{\bbS}Y_n^m(\hat\by) ds\\
=&-\frac{1}{R}\int_{\bbS_R} \nabla_\by\Phi(\bx-\by)\Delta_{\bbS}Y_n^m(\hat\by) ds
=n(n+1)\frac{1}{R}\int_{\bbS_R} \nabla_\by\Phi(\bx-\by)Y_n^m(\hat\by) ds.
\end{split}
\]
Therefore, the integral identity \eqnref{eq:imptidt01} follows from Proposition \ref{pr1} and the identity \eqnref{eq:fistdirG01}.
For the other integral identity, one has by direct calculations that
\[
\begin{split}
&\int_{\bbS_R} \nabla_\bx^2\Phi(\bx-\by)\cdot(Y_n^m(\hat\by)\bnu_{\by}) ds   \\
=&\int_{\bbS_R} \Big(\nabla_{\bbS_R}(\nabla_\by\Phi(\bx-\by))+\frac{\p}{\bnu_\by}(\nabla_\by\Phi(\bx-\by))\bnu_\by\Big)\cdot(Y_n^m(\hat\by)\bnu_{\by}) ds\\
=&\int_{\bbS_R} \frac{\p}{\bnu_\by}(\nabla_\by\Phi(\bx-\by))Y_n^m(\hat\by) ds.
\end{split}
\]
Finally, one can derive \eqnref{eq:imptidt02} from Proposition \ref{pr1} and the identity \eqnref{eq:normdirG01}.

The proof is complete.
\end{proof}

\begin{prop}\label{prusef02}
The following identity holds
\beq\label{eq:imptidt03}
\begin{split}
 \int_{\bbS_R} \nabla^2_\bx \Phi(\bx-\by)\cdot( \nabla_{\bbS}Y_n^m(\hat\by)\wedge \bnu_{\by}) ds =0.
\end{split}
\eeq
\end{prop}
\begin{proof}
By using integration by parts, there holds
\[
\begin{split}
&\int_{\bbS_R} \nabla^2_\bx \Phi(\bx-\by)\cdot( \nabla_{\bbS}Y_n^m(\hat\by)\wedge \bnu_{\by}) ds=\int_{\bbS_R} \nabla^2_\by \Phi(\bx-\by)\cdot( \nabla_{\bbS}Y_n^m(\hat\by)\wedge \bnu_{\by}) ds\\
=&\frac{1}{R}\int_{\bbS_R} \nabla_{\bbS}(\nabla_\by \Phi(\bx-\by))\cdot( \nabla_{\bbS}Y_n^m(\hat\by)\wedge \bnu_{\by}) ds\\
=&-\int_{\bbS_R}\nabla_\by \Phi(\bx-\by)\nabla_{\bbS}\cdot( \nabla_{\bbS}Y_n^m(\hat\by)\wedge \bnu_{\by}) ds=0,
\end{split}
\]
where the last identity follows from Lemma \ref{lem:1} and this completes the proof.
\end{proof}

With the above preparations, we are in a position to derive the spectral system of the single-layer potential operator $ \bS_{\bbS_R}^{\omega}$. To that end, we first show the following result about the single-layer potentials $ \bS_{\bbS_R}^{\omega}[\Tcal_n^m]$, $ \bS_{\bbS_R}^{\omega}[\Ical_{n-1}^m]$ and $  \bS_{\bbS_R}^{\omega}[\Ncal_{n+1}^m]$, which shall be critical to our subsequent analysis.

\begin{thm}\label{thm:singleo}
The single-layer potentials associated with the density functions $\Tcal_n^m$, $\Ical_{n-1}^m$ and $ \Ncal_{n+1}^m$ are given as follows for $\bx\in\mathbb{R}^3\backslash B_R$,
\[
\begin{split}
  \bS_{\bbS_R}^{\omega}[ \Ical_{n-1}^m](\bx) =&  -R^2 \rmi \left( \frac{ (n+1)k_s j_{n-1,s} h_{n-1}(k_s |\bx|) }{\mu(2n+1)}  +  \frac{ nk_p j_{n-1,p} h_{n-1}(k_p |\bx|)}{(\lambda+2\mu)(2n+1)}   \right) \Ical_{n-1}^m \\
       &- n R^2\rmi  \left( \frac{k_s j_{n-1,s} h_{n+1}(k_s |\bx|) }{\mu(2n+1)}  -   \frac{k_p j_{n-1,p} h_{n+1}(k_p |\bx|) }{(\lambda+2\mu)(2n+1)}  \right) \Ncal_{n+1}^m,
 \end{split}
\]

\[
\begin{split}
  \bS_{\bbS_R}^{\omega}[ \Ncal_{n+1}^m](\bx) =&  - (n+1) R^2\rmi  \left( \frac{ k_s j_{n+1,s} h_{n-1}(k_s |\bx|)}{\mu(2n+1)}  -  \frac{ k_p j_{n+1,p} h_{n-1}(k_p |\bx|) }{(\lambda+2\mu)(2n+1)}   \right) \Ical_{n-1}^m \\
       &-R^2 \rmi \left( \frac{ n k_s j_{n+1,s} h_{n+1}(k_s |\bx|) }{\mu(2n+1)}  +  \frac{ (n+1) k_p j_{n+1,p} h_{n+1}(k_p |\bx|)}{(\lambda+2\mu)(2n+1)}  \right) \Ncal_{n+1}^m,
 \end{split}
\]
and
\[
 \bS_{\bbS_R}^{\omega}[\Tcal_n^m](\bx) = -\frac{\rmi k_s R^2 j_{n,s} h_n(k_s|\bx|)}{\mu}\Tcal_n^m,
\]
where and also in what follows, we denote $j_{n}(k_s R)$, $j_{n}(k_p R)$, $h_{n}(k_s R)$ and $h_{n}(k_p R)$ by  $j_{n,s}$, $j_{n,p}$, $h_{n,s}$ and $h_{n,p}$ for simplicity.
\end{thm}
\begin{proof}	
The proof follows from the expression of the fundamental solution $\bGa^{\omega}$ defined in \eqref{eq:ef}, Lemma \ref{lem:sk}, and Propositions \ref{prusef01} and \ref{prusef02}, along with straightforward (though tedious) calculations.
\end{proof}

By a similar argument to Theorem \ref{thm:singleo}, one can show
\begin{prop}\label{pro:singlei}
For $\bx\in B_R$, the single-layer potentials $ \bS_{\bbS_R}^{\omega}[\Tcal_n^m]$, $ \bS_{\bbS_R}^{\omega}[\Ical_{n-1}^m]$ and $  \bS_{\bbS_R}^{\omega}[\Ncal_{n+1}^m]$  are given as follows
\[
\begin{split}
  \bS_{\bbS_R}^{\omega}[ \Ical_{n-1}^m](\bx) =&  -R^2 \rmi \left( \frac{ (n+1)k_s h_{n-1,s} j_{n-1}(k_s |\bx|) }{\mu(2n+1)}  +  \frac{ nk_p h_{n-1,p} j_{n-1}(k_p |\bx|)}{(\lambda+2\mu)(2n+1)}   \right) \Ical_{n-1}^m \\
       &- n R^2\rmi  \left( \frac{k_s h_{n-1,s} j_{n+1}(k_s |\bx|) }{\mu(2n+1)}  -   \frac{k_p h_{n-1,p} j_{n+1}(k_p |\bx|) }{(\lambda+2\mu)(2n+1)}  \right) \Ncal_{n+1}^m,
 \end{split}
\]

\[
\begin{split}
  \bS_{\bbS_R}^{\omega}[ \Ncal_{n+1}^m](\bx) =&  - (n+1) R^2\rmi  \left( \frac{ k_s h_{n+1,s} j_{n-1}(k_s |\bx|)}{\mu(2n+1)}  -  \frac{ k_p h_{n+1,p} j_{n-1}(k_p |\bx|) }{(\lambda+2\mu)(2n+1)}   \right) \Ical_{n-1}^m \\
       &-R^2 \rmi \left( \frac{ n k_s h_{n+1,s} j_{n+1}(k_s |\bx|) }{\mu(2n+1)}  +  \frac{ (n+1) k_p h_{n+1,p} j_{n+1}(k_p |\bx|)}{(\lambda+2\mu)(2n+1)}  \right) \Ncal_{n+1}^m,
 \end{split}
\]
and
\[
 \bS_{\bbS_R}^{\omega}[\Tcal_n^m](\bx) = -\frac{\rmi k_s R^2 h_{n,s} j_n(k_s|\bx|)}{\mu}\Tcal_n^m.
\]

\end{prop}

From Theorem \ref{thm:singleo} and the continuity of the single layer potential operator $\bS_{\bbS_R}^{\omega}$ from $\bx\in\mathbb{R}^3\backslash B_R$ to $\bx\in\mathbb{S}_R$, one can conclude that for $\bx\in\bbS_R$
\begin{equation}\label{eq:sint}
\bS_{\bbS_R}^{\omega}[\Tcal_n^m](\bx) = b_n\Tcal_n^m,\quad
 \bS_{\bbS_R}^{\omega}[ \Ical_{n-1}^m](\bx) = c_{1n} \Ical_{n-1}^m + d_{1n} \Ncal_{n+1}^m,
\end{equation}
and
\begin{equation}
 \bS_{\bbS_R}^{\omega}[ \Ncal_{n+1}^m](\bx) = c_{2n} \Ical_{n-1}^m + d_{2n}\Ncal_{n+1}^m.
\end{equation}
where
\[
\begin{split}
b_n=&  -\frac{\rmi k_s R^2 j_{n,s} h_{n,s}}{\mu},\\
c_{1n}=& -R^2 \rmi \left( \frac{j_{n-1,s} h_{n-1,s} k_s (n+1)}{\mu(2n+1)}  +  \frac{j_{n-1,p} h_{n-1,p} k_p n}{(\lambda+2\mu)(2n+1)}   \right),\\
d_{1n}=& - n R^2\rmi  \left( \frac{j_{n-1,s} h_{n+1,s} k_s}{\mu(2n+1)}  -   \frac{j_{n-1,p} h_{n+1,p} k_p }{(\lambda+2\mu)(2n+1)}  \right),\\
c_{2n}=& - (n+1) R^2\rmi  \left( \frac{j_{n+1,s} h_{n-1,s} k_s}{\mu(2n+1)}  -  \frac{j_{n+1,p} h_{n-1,p} k_p }{(\lambda+2\mu)(2n+1)}   \right),\\
d_{2n}=&  -R^2 \rmi \left( \frac{j_{n+1,s} h_{n+1,s} k_s n}{\mu(2n+1)}  +  \frac{j_{n+1,p} h_{n+1,p} k_p (n+1)}{(\lambda+2\mu)(2n+1)}  \right).
\end{split}
\]

The rest of the section is devoted to the derivation of the traction of the single layer potential on the $\mathbb{S}_R$, based on which, we can derive the spectral system of the N-P operator. First of all, we deduce the following two propositions.
\begin{prop}\label{pro:tra1}
 The following identities hold for $n,p\in\mathbb{N}_0$:
 \[
  \nabla\cdot \left( h_n(k|\bx|)\nabla_{\bbS}Y_p^m \right) =-p(p+1)h_n(k|\bx|)Y_p^m/|\bx|,
 \]
 \[
  \nabla\cdot \left( h_n(k|\bx|)Y_p^m\bnu \right) =(k h_n^{\prime}(k|\bx|)  + 2 h_n(k|\bx|)/|\bx| )Y_p^m,
 \]
 and
 \[
   \nabla\cdot \left( h_n(k|\bx|)\nabla_{\bbS}Y_{p}^m\wedge\bnu \right)=0.
 \]
\end{prop}
\begin{proof}
By the vector calculus identity, one has that
\[
 \begin{split}
 & \nabla\cdot \left( h_n(k|\bx|)\nabla_{\bbS}Y_p^m \right)=\nabla h_n(k|\bx|)\cdot \nabla_{\bbS}Y_p^m + h_n(k|\bx|)  \nabla\cdot  \nabla_{\bbS}Y_p^m \\
  =& h_n(k|\bx|) \triangle_{\bbS}Y_p^m/|\bx| = -p(p+1)h_n(k|\bx|)Y_p^m /|\bx|,
 \end{split}
\]
where the last two identities follow from Lemmas \ref{lem:1} and \ref{lem:2}. Therefore, one can show the first identity of the proposition. The other two identities of the proposition can be shown in a similar manner.
\end{proof}

\begin{prop}\label{pro:tra2}
The following identities hold for $n,p\in\mathbb{N}_0$:
\[
\begin{split}
 \nabla\left( h_n(k|\bx|)\nabla_{\bbS}Y_p^m \right)\bnu = &k h_n^{\prime}(k|\bx|) \nabla_{\bbS}Y_p^m,\\
 \nabla\left( h_n(k|\bx|)\nabla_{\bbS}Y_p^m \right)^T \bnu =& -h_n(k|\bx|) \nabla_{\bbS}Y_p^m/|\bx|,\\
 \nabla\left( h_n(k|\bx|)Y_p^m\bnu \right)\bnu =& k h_n^{\prime}(k|\bx|) Y_p^m\bnu,\\
 \nabla\left( h_n(k|\bx|)Y_p^m\bnu \right)^T \bnu =&k h_n^{\prime}(k|\bx|)Y_p^m\bnu  + h_n(k|\bx|)/|\bx| \nabla_{\bbS}Y_p^m,\\
 \nabla\left( h_n(k|\bx|)\nabla_{\bbS}Y_{p}^m\wedge\bnu \right)\bnu =& k h_n^{\prime}(k|\bx|) \nabla_{\bbS}Y_{p}^m\wedge\bnu,\\
 \nabla\left( h_n(k|\bx|)\nabla_{\bbS}Y_{p}^m\wedge\bnu \right)^T \bnu =& -h_n(k|\bx|)/|\bx| \nabla_{\bbS}Y_{p}^m\wedge\bnu.
\end{split}
\]
\end{prop}

\begin{proof}
In the following, we only give the proof of the first two identities and the other ones can be proved in a similar manner. First, one has
\[
\begin{split}
 \nabla\left( h_n(k|\bx|)\nabla_{\bbS}Y_p^m \right)\bnu&=(\nabla_{\bbS}Y_p^m \nabla h_n(k|\bx|)^T + h_n(k|\bx|)\nabla \nabla_{\bbS}Y_p^m)\bnu  = k h_n^{\prime}(k|\bx|) \nabla_{\bbS}Y_p^m,
\end{split}
\]
where the last identity follows from the following fact
\begin{equation}\label{eq:or1}
 (\nabla \nabla_{\bbS}Y_p^m)\bnu=\left(\frac{1}{|\bx|}\nabla_{\bbS} \nabla_{\bbS}Y_p^m\right)\bnu =0.
\end{equation}
Noting the symmetry of $ \nabla \nabla Y_p^m$ and rewriting \eqref{eq:or1} as
\[
 (\nabla \nabla_{\bbS}Y_p^m)\bnu= \left(\nabla \left(|\bx| \nabla Y_p^m\right) \right) \bnu =  \left(|\bx|\nabla \nabla Y_p^m +  \nabla_{\bbS}Y_p^m\bnu^T \right) \bnu =0,
\]
one can obtain that
\begin{equation}\label{eq:or2}
 \nabla \nabla Y_p^m \bnu =  \left(\nabla \nabla Y_p^m\right)^T \bnu =- \nabla_{\bbS}Y_p^m/|\bx|.
\end{equation}
Similarly one has that
\[
\begin{split}
 & \nabla\left( h_n(k|\bx|)\nabla_{\bbS}Y_p^m \right)^T\bnu=\left(\nabla h_n(k|\bx|) (\nabla_{\bbS}Y_p^m)^T + h_n(k|\bx|) (\nabla \nabla_{\bbS}Y_p^m)^T \right)\bnu \\
 =& h_n(k|\bx|)  \left( \left(\nabla \nabla Y_p^m\right)^T + \bnu\left( \nabla_{\bbS}Y_p^m\right)^T \right) \bnu = -h_n(k|\bx|) \nabla_{\bbS}Y_p^m/|\bx|,
\end{split}
\]
where the last identity follows from \eqref{eq:or2}. Hence, we have shown the first two identities.

The proof is complete.
\end{proof}

Next, we derive the tractions of the single-layer potentials $ \bS_{\bbS_R}^{\omega}[\Tcal_n^m]$, $ \bS_{\bbS_R}^{\omega}[\Ical_{n-1}^m]$ and $  \bS_{\bbS_R}^{\omega}[\Ncal_{n+1}^m]$ on $\bbS_R$.

\begin{prop}\label{pro:trac}
The traction of the single layer potentials $ \bS_{\bbS_R}^{\omega}[\Tcal_n^m]$, $ \bS_{\bbS_R}^{\omega}[\Ical_{n-1}^m]$ and $  \bS_{\bbS_R}^{\omega}[\Ncal_{n+1}^m]$ on $\bbS_R$ satisfy
\begin{eqnarray}
\partial_{\bnu}  \bS_{\bbS_R}^{\omega}[\Tcal_n^m](\bx) & = & \mathfrak{b}_n\Tcal_n^m,\label{eq:tract}\\
\partial_{\bnu}  \bS_{\bbS_R}^{\omega}[ \Ical_{n-1}^m] |_{+}(\bx) & = & \mathfrak{c}_{1n} \Ical_{n-1}^m +  \mathfrak{d}_{1n} \Ncal_{n+1}^m,\label{eq:traci}\\
\partial_{\bnu}  \bS_{\bbS_R}^{\omega}[ \Ncal_{n+1}^m] |_{+}(\bx) & = & \mathfrak{c}_{2n} \Ical_{n-1}^m +  \mathfrak{d}_{2n} \Ncal_{n+1}^m,\label{eq:tracn}
\end{eqnarray}
where
\begin{equation}\label{eq:cotrac}
\begin{split}
 \mathfrak{b}_n=&  -\rmi k_s R j_{n,s} (k_sR h_{n,s}^{\prime} - h_{n,s}), \\
 \mathfrak{c}_{1n}=& -2(n-1)R\rmi  \left( \frac{j_{n-1}(k_s R) h_{n-1}(k_s R) k_s (n+1)}{2n+1}  +  \frac{j_{n-1}(k_p R) h_{n-1}(k_p R) k_p\mu n}{(\lambda+2\mu)(2n+1)}   \right)\\
         & + R^2\rmi  \left( \frac{ j_{n-1}(k_s R) h_{n}(k_s R) k_s^2  (n+1) + j_{n-1}(k_p R) h_{n}(k_p R) k_p^2 n}{2n+1}   \right), \\
\mathfrak{d}_{1n}=&  2n(n+2) R\rmi  \left( \frac{j_{n-1}(k_s R) h_{n+1}(k_s R) k_s}{2n+1}  -   \frac{j_{n-1}(k_p R) h_{n+1}(k_p R) k_p\mu }{(\lambda+2\mu)(2n+1)}  \right)\\
        & +  n R^2\rmi  \left( \frac{ -j_{n-1}(k_s R) h_{n}(k_s R) k_s^2  + j_{n-1}(k_p R) h_{n}(k_p R) k_p^2}{2n+1}   \right), \\
\mathfrak{c}_{2n}=& - 2(n^2-1) R \rmi  \left( \frac{j_{n+1}(k_s R) h_{n-1}(k_s R) k_s}{2n+1}  -  \frac{j_{n+1}(k_p R) h_{n-1}(k_p R) k_p\mu }{(\lambda+2\mu)(2n+1)}   \right)\\
         & -(n+1) R^2\rmi  \left( \frac{ -j_{n-1}(k_s R) h_{n}(k_s R) k_s^2   + j_{n-1}(k_p R) h_{n}(k_p R) k_p^2 }{2n+1}   \right), \\
\mathfrak{d}_{2n}=&  2(n+2)R \rmi  \left( \frac{j_{n+1}(k_s R) h_{n+1}(k_s R) k_s n}{(2n+1)}  +  \frac{j_{n+1}(k_p R) h_{n+1}(k_p R) k_p \mu (n+1)}{(\lambda+2\mu)(2n+1)}  \right)\\
         & - R^2\rmi  \left( \frac{ j_{n+1}(k_s R) h_{n}(k_s R) k_s^2 n  + j_{n+1}(k_p R) h_{n}(k_p R) k_p^2(n+1) }{2n+1}   \right).
\end{split}
\end{equation}
\end{prop}

\begin{proof}
The proof follows from straightforward though tedious calculations along with the help of \eqref{eq:trac} and Propositions \ref{pro:tra1} and \ref{pro:tra2}.
\end{proof}

We are in a position to present the spectral system of the N-P operator $\left(\bK_{\bbS_R}^{\omega} \right)^* $.
\begin{thm}\label{thm:ks}
The spectral system of the N-P operator  $\left(\bK_{\bbS_R}^{\omega} \right)^*$ is given as follows
\begin{eqnarray}
\left(\bK_{\bbS_R}^{\omega} \right)^*[\Tcal_n^m]&=&\lambda_{1,n} \Tcal_n^m,\label{eq:kei1}\\
  \left(\bK_{\bbS_R}^{\omega} \right)^*[\Ucal_n^m]&=&\lambda_{2,n} \Ucal_n^m,\label{eq:kei2}\\
  \left(\bK_{\bbS_R}^{\omega} \right)^*[\Vcal_n^m]&=&\lambda_{3,n} \Vcal_n^m,\label{eq:kei3}
\end{eqnarray}
where
\[
 \lambda_{1,n} =\mathfrak{b}_n-1/2,
\]
and if $\mathfrak{d}_{1n}\neq 0$,
\[
\begin{split}
 \lambda_{2,n} =& \frac{\mathfrak{c}_{1n} + \mathfrak{d}_{2n} -1  + \sqrt{(\mathfrak{d}_{2n}-\mathfrak{c}_{1n})^2 + 4 \mathfrak{d}_{1n} \mathfrak{c}_{2n}}}{2},\\
 \lambda_{3,n} =& \frac{\mathfrak{c}_{1n} + \mathfrak{d}_{2n}  -1 - \sqrt{(\mathfrak{d}_{2n}-\mathfrak{c}_{1n})^2 + 4 \mathfrak{d}_{1n} \mathfrak{c}_{2n}}}{2},\\
 \Ucal_n^m=& \left(\mathfrak{c}_{1n} - \mathfrak{d}_{2n}  + \sqrt{(\mathfrak{d}_{2n}-\mathfrak{c}_{1n})^2 + 4 \mathfrak{d}_{1n} \mathfrak{c}_{2n}}\right) \Ical_{n-1}^m  + 2\mathfrak{d}_{1n} \Ncal_{n+1}^m,\\
 \Vcal_n^m=& \left(\mathfrak{c}_{1n} - \mathfrak{d}_{2n}  - \sqrt{(\mathfrak{d}_{2n}-\mathfrak{c}_{1n})^2 + 4 \mathfrak{d}_{1n} \mathfrak{c}_{2n}}\right) \Ical_{n-1}^m  + 2\mathfrak{d}_{1n} \Ncal_{n+1}^m;
 \end{split}
\]
if $\mathfrak{d}_{1n}= 0$,
\[
\begin{split}
 \lambda_{2,n} =& \mathfrak{c}_{1n}-1/2, \quad  \lambda_{3,n} =\mathfrak{d}_{2n}-1/2,\\
 \Ucal_n^m= &\Ical_{n-1}^m , \quad  \Vcal_n^m= \mathfrak{c}_{2n}  \Ical_{n-1}^m  + (\mathfrak{d}_{2n}-\mathfrak{c}_{1n} )\Ncal_{n+1}^m,
\end{split}
\]
with $ \Tcal_n^m$, $\Ical_{n}^m$ and $\Ncal_{n}^m$ given in Lemma \ref{lem:vec}, and the parameters $\mathfrak{b}_n$, $\mathfrak{c}_{1n}$, $\mathfrak{d}_{1n}$, $\mathfrak{c}_{2n}$ and $\mathfrak{d}_{2n}$ defined in \eqref{eq:cotrac}.
\end{thm}
\begin{proof}
From the the jump formula \eqref{eq:jump} and the identity \eqref{eq:tract}, one can directly have that
\[
  \left(\bK_{\bbS_R}^{\omega} \right)^*[\Tcal_n^m]=  \frac{\partial }{\partial \bnu}  \bS_{\bbS_R}^{\omega}[\Tcal_n^m] -\frac{1}{2} \Tcal_n^m =(\mathfrak{b}_n-1/2) \Tcal_n^m.
\]
Hence, the first identity \eqref{eq:kei1} is proved. For the other two ones, namely \eqref{eq:kei2} and \eqref{eq:kei3}, by noting \eqref{eq:traci} and \eqref{eq:tracn}, one sees that the eigenfunctions should be the linear combinations of $\Ical_{n-1}^m$ and $\Ncal_{n+1}^m$. Hence, we can assume that the eigenfunctions have the following form $a \Ical_{n-1}^m + \Ncal_{n+1}^m$, namely,
\begin{equation}\label{eq:ei0}
   \left(\bK_{\bbS_R}^{\omega} \right)^* [a \Ical_{n-1}^m + \Ncal_{n+1}^m] =\lambda (a \Ical_{n-1}^m + \Ncal_{n+1}^m).
\end{equation}
Again from the jump formula \eqref{eq:jump} and the identities \eqref{eq:traci} and \eqref{eq:tracn}, one has that
\begin{equation}\label{eq:ki}
\begin{split}
   \left(\bK_{\bbS_R}^{\omega} \right)^* [ \Ical_{n-1}^m ]=& (\mathfrak{c}_{1n}-1/2) \Ical_{n-1}^m + \mathfrak{d}_{1n} \Ncal_{n+1}^m,\\
  \left(\bK_{\bbS_R}^{\omega} \right)^* [ \Ncal_{n+1}^m ] =& \mathfrak{c}_{2n} \Ical_{n-1}^m + (\mathfrak{d}_{2n}-1/2) \Ncal_{n+1}^m.
\end{split}
\end{equation}
Substituting the last two equations into \eqref{eq:ei0} and comparing the coefficient on both sides yield that
\begin{equation}\label{eq:2ti}
  a^2 \mathfrak{d}_{1n} + a(\mathfrak{d}_{2n}-\mathfrak{c}_{1n})-\mathfrak{c}_{2n}=0.
\end{equation}
If $\mathfrak{d}_{1n}\neq0$,  solving the equation \eqref{eq:2ti} gives that
\begin{equation}\label{eq:so1}
 a=\frac{\mathfrak{c}_{1n} -\mathfrak{d}_{2n} \pm \sqrt{(\mathfrak{d}_{2n}-\mathfrak{c}_{1n})^2+4\mathfrak{d}_{1n}\mathfrak{c}_{2n}}}{2\mathfrak{d}_{1n}}.
\end{equation}
Therefore, the two identities \eqref{eq:kei2} and \eqref{eq:kei3}, follow from substituting \eqref{eq:so1} into \eqref{eq:ei0}.
If $\mathfrak{d}_{1n}=0$, from the equation \eqref{eq:ki}, one can directly have that
\[
 \left(\bK_{\bbS_R}^{\omega} \right)^* [ \Ical_{n-1}^m ]= (\mathfrak{c}_{1n}-1/2) \Ical_{n-1}^m ,
\]
which signifies that $\Ical_{n-1}^m$ is one of the eigenfunctions of the N-P operator $\bK_{\bbS_R}^{\omega}$ corresponding to the eigenvalue $\mathfrak{c}_{1n}-1/2$. For the other eigenfunction containing $\Ncal_{n+1}^m$, solving the equation \eqref{eq:2ti} yields that
\[
 a= \frac{\mathfrak{c}_{2n}}{\mathfrak{d}_{2n}-\mathfrak{c}_{1n}}.
\]
Substituting the last equation into \eqref{eq:ei0} yields that
\[
 \left(\bK_{\bbS_R}^{\omega} \right)^* [  \Vcal_n^m ] =  (\mathfrak{d}_{2n}-1/2)  \Vcal_n^m,
\]
where
\[
  \Vcal_n^m= \mathfrak{c}_{2n}  \Ical_{n-1}^m  + (\mathfrak{d}_{2n}-\mathfrak{c}_{1n} )\Ncal_{n+1}^m.
\]

The proof is complete.
\end{proof}
\begin{rem}\label{rem:sta}
By taking $\omega\rightarrow +0$ in the spectral results in Theorem \ref{thm:ks} and applying the asymptotic properties of the spherical Bessel and Hankel functions, $j_n(t)$ and $h_n(t)$, for $t\ll1$ in \eqref{eq:asjt}, one can obtain after straightforward though tedious calcuations the spectral system of the N-P operator $\bK_{\bbS_R}^{0}$ in the static case, which coincides with that established in \cite{DLL}.
\end{rem}

\section{Polariton resonance beyond the quasi-static approximation}
In this section, using the spectral results established in the previous section, we construct a broad class of elastic structures of the form $\mathbf{C}_0$ in \eqref{eq:pa1} with no core, namely $D=\emptyset$ that can induce polariton resonances. Suppose that a source term $\bff\in H^{-1}(\mathbb{R}^3)^3$ is compactly supported outside $\Omega$, then the elastic system \eqref{eq:lame1} can be simplified as the following transmission problem

\begin{equation}\label{eq:nocore}
  \left\{
    \begin{array}{ll}
      \mathcal{L}_{\hat{\lambda}, \hat{\mu}}\bu(\bx) + \omega^2\bu(\bx) =0,    & \bx\in \Omega \\
      \mathcal{L}_{\lambda, \mu}\bu(\bx) + \omega^2\bu(\bx) =\bff, & \bx\in \mathbb{R}^3\backslash \overline{\Omega}\\
      \bu(\bx)|_- = \bu(\bx)|_+,      & \bx\in\partial \Omega \\
      \partial_{\hat{\bnu}}\bu(\bx)|_- = \partial_{\bnu}\bu(\bx)|_+, & \bx\in\partial \Omega,
    \end{array}
  \right.
\end{equation}
where $\partial_{\bnu}$ is given in \eqref{eq:trac}, $\mathcal{L}_{\lambda, \mu}$ is defined in \eqref{op:lame} and $\bu$ satisfies the radiation condition \eqref{eq:radi}. In \eqref{eq:nocore} and also in what follows, $ \mathcal{L}_{\hat{\lambda}, \hat{\mu}}$ and $ \partial_{\hat{\bnu}}$ denote the Lam\'e operator and the traction operator associated with the Lam\'e parameters $\hat{\lambda}$ and $\hat{\mu}$, and the same notations hold for the single-layer potential operator $\hat{\bS}_{\Omega}^{\omega}$ and the N-P operator $(\hat{\bK}_{\partial\Omega}^{\omega})^*$.

Using the single-layer potential defined in \eqref{eq:single}, the solution to the system \eqref{eq:nocore} can be written as
\begin{equation}\label{eq:sol}
  \bu=
 \left\{
   \begin{array}{ll}
     \hat{\bS}_{\partial\Omega}^{\omega}[\bpsi_1](\bx), & \bx\in \Omega, \\
     \bS_{\partial\Omega}^{\omega}[\bpsi_2](\bx) + \mathbf{F}, &  \bx\in \mathbb{R}^3\backslash \overline{\Omega},
   \end{array}
 \right.
\end{equation}
where
\begin{equation}\label{eq:newpf}
 \mathbf{F}(\bx):= \int_{\mathbb{R}^3} \mathbf{\Gamma}^{\omega}(\bx-\by)\bff(\by)d\by, \quad \bx\in \mathbb{R}^3,
\end{equation}
is called the Newtonian potential of the source $\bff$ and $\bpsi_1, \bpsi_2\in L^2(\partial\Omega)^3$. One can readily verify that the solution defined in \eqref{eq:sol} satisfy the first two conditions in \eqref{eq:nocore}. For the third and forth condition in \eqref{eq:nocore} across $\partial\Omega$, namely the transmission condition, one can obtain that
\begin{equation}\label{eq:sol1}
  \left\{
    \begin{array}{ll}
      \hat{\bS}_{\partial\Omega}^{\omega}[\bpsi_1] - \bS_{\partial\Omega}^{\omega}[\bpsi_2] = \mathbf{F}, \\
      \partial_{\hat{\bnu}}\hat{\bS}_{\partial\Omega}^{\omega}[\bpsi_1]|_- - \partial_{\bnu}\bS_{\partial\Omega}^{\omega}[\bpsi_2]|_+ = \partial_{\bnu}\mathbf{F} ,
    \end{array}
  \right.
  \quad \bx\in\partial \Omega.
\end{equation}
With the help of the jump formula \eqref{eq:jump}, the equation \eqref{eq:sol1} can be rewritten as
\begin{equation}\label{eq:ma1}
  \bA^{\omega}
 \left[
   \begin{array}{c}
     \bpsi_1 \\
     \bpsi_2 \\
   \end{array}
 \right]=
\left[
  \begin{array}{c}
    \mathbf{F} \\
    \partial_{\bnu}\mathbf{F} \\
  \end{array}
\right],
\end{equation}
where
\begin{equation}
  \bA^{\omega} =
 \left[
   \begin{array}{cc}
      \hat{\bS}_{\partial\Omega}^{\omega} & - \bS_{\partial\Omega}^{\omega} \\
     -1/2I+(\hat{\bK}_{\partial\Omega}^{\omega})^* & -1/2I-(\bK_{\partial\Omega}^{\omega})^* \\ \\
   \end{array}
 \right].
\end{equation}
In the following, we assume that the domain $\Omega$ is a ball $B_R$. Since the source term $\bff$ is supported outside $B_R$, there exists $\epsilon>0$ such that when $\bx\in B_{R+\epsilon}$, the Newtonian potential $\bF$ defined in \eqref{eq:newpf} satisfies
\[
 \Lcal_{\lambda,\mu}\bF + \omega^2\bF=0.
\]
Thus $\bF$ can be written as
\begin{equation}\label{eq:FF}
  \bF= \sum_{n=0}^{\infty} \sum_{m=-n}^{n} \left(f_{1,n,m} j_n(k_s|\bx|) \Tcal_n^m + f_{2,n,m}  \bS_{\bbS_R}^{\omega}[ \Ical_{n-1}^m]  + f_{3,n,m} \bS_{\bbS_R}^{\omega}[ \Ncal_{n+1}^m] \right) ,
\end{equation}
for $\bx\in B_{R+\epsilon}$, which follows from Lemma \ref{lem:vec} and Proposition \ref{pro:singlei}.

Our main result in this section is stated in the following theorem. It characterizes the polariton resonance for the configuration without a core.

\begin{thm}\label{thm:reson}
Consider the configuration $\mathbf{C}_0$ with $D=\emptyset$ defined in \eqref{eq:pa1}. Suppose that the source term $\bff\in H^{-1}(\mathbb{R}^3)^3$ is compactly supported outside the domain $\Omega$, whose  Newtonian potential $\bF$ is defined in \eqref{eq:FF} with $f_{1,n_0,m}\neq 0$ for  some $n_0\in\mathbb{N}$.  For any $M\in\mathbb{R}_+$, if the Lam\'e parameter $\hat{\mu}$ inside the domain $\Omega$ is chosen such that
\begin{equation}\label{con:pla1}
  \frac{ \Im(\hat{\mu})}{|\widetilde{\psi}_{1,n_0,m}|^2}>M,
\end{equation}
where $\widetilde{\psi}_{1,n_0,m}$ is defined in \eqref{eq:psi10}, then the polariton resonance occurs.

Furthermore, if $n_0\gg 1$ is large enough such that the spherical Bessel and Hankel functions, $j_n(t)$ and $h_n(t)$, enjoy the asymptotic expression shown in \eqref{eq:asj}, then one can choose the Lam\'e parameter $\hat{\mu}$ inside the domain $\Omega$ as follows
\begin{equation}\label{con:001}
\hat{\mu}=-\mu+ \rmi \frac{1}{M}+ p_{1,n_0},
\end{equation}
where $p_{1,n_0}$ should satisfy
\begin{equation}\label{con:n01}
 p_{1,n_0} + q_{1,n_0}=\Ocal\left(\frac{1}{M}\right),
\end{equation}
with $q_{1,n_0}$ defined in \eqref{eq:coe1}, to ensure the occurrence of the polariton resonance.
\end{thm}

%\begin{thm}\label{thm:reson}
%Consider the configuration $\mathbf{C}_0$ with $D=\emptyset$ defined in \eqref{eq:pa1} and a source term $\bff$ supported outside the domain $\Omega$.
%If the Lam\'e parameter $\hat{\mu}$ inside the domain $\Omega$ is chosen such that for any $M\in\mathbb{R}_+$
%\begin{equation}\label{con:pla1}
%\color{red}  \frac{ \Im(\hat{\mu})}{|\widetilde{\psi}_{1,n_0,m}|^2}>M,
%\end{equation}
%for some $n_0\in\mathbb{N}$ with $f_{1,n_0,m}\neq 0$, where $\widetilde{\psi}_{1,n_0,m}$ is defined in \eqref{eq:psi10} and $f_{1,n_0,m}$ is given in \eqref{eq:FF}, then polariton resonance occurs.
%
%Futhermore, if $n_0\gg 1$ is large enough such that the spherical Bessel and Hankel functions, $j_n(t)$ and $h_n^{(1)}(t)$, enjoy the asymptotic expression shown in \eqref{eq:asj}, then one can choose the Lam\'e parameter $\hat{\mu}$ inside the domain $\Omega$ as follows
%\begin{equation}\label{con:001}
%\color{red} \hat{\mu}=-\mu+ i \delta + p_{1,n_0},
%\end{equation}
%where $p_{1,n_0}$ should satisfy
%\begin{equation}\label{con:n01}
% p_{1,n_0} + q_{1,n_0}=\Ocal(\delta),
%\end{equation}
%with $q_{1,n_0}$ defined in \eqref{eq:coe1}, to ensure the occurrence of the polariton resonance.
%\end{thm}

\begin{proof}
Following Propositions \ref{pro:singlei} and \ref{pro:trac}, one can conclude that the displacement and traction of the term $j_n(k_s|\bx|) \Tcal_n^m$ on the boundary $B_R$ are orthogonal to both the corresponding components of the other two terms, namely $\bS_{\bbS_R}^{\omega}[ \Ical_{n-1}^m] $ and $\bS_{\bbS_R}^{\omega}[ \Ncal_{n-1}^m] $. Therefore, in order to show the polariton resonance, it suffices to consider the source only containing the terms $j_n(k_s|\bx|) \Tcal_n^m$, namely
\begin{equation}\label{eq:FF1}
  \bF= \sum_{n=0}^{\infty} \sum_{m=-n}^{n} \left(f_{1,n,m} j_n(k_s|\bx|) \Tcal_n^m \right) .
\end{equation}
Thanks to the orthogonality of the functions $\Tcal_n^m$, $\Ical_n^m$ and $\Ncal_n^m$, the density functions in \eqref{eq:sol} have the following expressions
\begin{equation}
\begin{split}
\bpsi_1&=\sum_{n=0}^{+\infty}\sum_{m=-n}^n \psi_{1,n,m} \Tcal_n^m, \\
\bpsi_2&=\sum_{n=0}^{+\infty}\sum_{m=-n}^n \psi_{2,n,m} \Tcal_n^m.
\end{split}	
\end{equation}
From the jump formula \eqref{eq:jump}, and Propositions \ref{pro:singlei} and  \ref{pro:trac},  the equation \eqref{eq:ma1} can be written as
\begin{equation}\label{eq:matrix}
 \left[
   \begin{array}{cc}
      a_{11} &  a_{12} \\
    a_{21}  &  a_{22} \\
   \end{array}
 \right]
 \left[
   \begin{array}{c}
     \psi_{1,n,m} \\
     \psi_{2,n,m} \\
   \end{array}
 \right]=
\left[
  \begin{array}{c}
    f_{1,n,m} j_n(k_sR) \\
    g_{1,n,m} \\
  \end{array}
\right],
\end{equation}
where
\[
 a_{11}=-\frac{\rmi \hat{k}_s R^2 j_n(\hat{k}_s R) h_n(\hat{k}_s R)}{\mu}, \quad a_{12}=\frac{\rmi k_s R^2 j_n(k_s R) h_n(k_s R)}{\mu},
\]
\[
 a_{21}=-\rmi \hat{k}_s R^2 h_n(\hat{k}_s R) \left(\hat{k}_s R  j_n^{\prime}(\hat{k}_s R) - j_n(\hat{k}_s R) \right),
 \]
\[
 a_{22}= -1 + \rmi k_s R^2 h_n(k_s R) \left(k_s R  j_n^{\prime}(k_s R) - j_n(k_s R) \right),
\]
and
\[
 g_{1,n,m}= f_{1,n,m}\mu  \left(k_s R  j_n^{\prime}(k_s R) - j_n(k_s R) \right)/R,
\]
with $\hat{k}_s=\omega/\sqrt{\hat{\mu}}$. With the help of the Wronskian identity
\[
  j_n(t)h_n^{\prime}(t)-j_n^{\prime}(t)h_n(t)=\frac{\rmi}{t^2}, \quad \mbox{for} \quad t>0,
\]
solving the equation \eqref{eq:matrix} yields that
\begin{equation}\label{eq:psi10}
 \psi_{1,n,m}=\frac{f_{1,n,m} j_n(k_sR)}{ \widetilde{\psi}_{1,n,m} },
\end{equation}
where
\[
\begin{split}
 \widetilde{\psi}_{1,n,m} = & \left( ((\mu-\hat{\mu})j_n(\hat{k}_sR) + \hat{k}_s\hat{mu}R j_n^{\prime}(\hat{k}_sR))h_n(k_s R)  \right. \\
  & \left. -k_s \mu R j_n(\hat{k}_sR)h_n^{\prime}(k_sR)  \right) k_s \hat{k}_s R^3 j_n(k_s R)h_n(\hat{k}_sR).
 \end{split}
\]
Next we calculate the dissipation energy $E(\bu)$. From the definition of the functional $P_{\lambda,\mu}(\bu,\bu)$ given in \eqref{eq:func1} and the following identity
\[
 \nabla\cdot\bu= \nabla\cdot\hat{\bS}_{\partial\Omega}^{\omega}[\Tcal_n^m]=0,
\]
 there holds that
\begin{equation}\label{eq:E1}
\begin{split}
 E(\bu) &=  \Im P_{\hat{\lambda},\hat{\mu}}(\bu, \bu) =  \Im\left( \hat{\mu}  P_{\hat{\lambda}/\hat{\mu},1}(\bu, \bu) \right)  \\
 &=\Im(\hat{\mu}) \sum_{n=0}^{+\infty}\sum_{m=-n}^n\left( |\psi_{1,n,m}|^2  P_{\hat{\lambda}/\hat{\mu},1}\left( \hat{\bS}_{\partial\Omega}^{\omega}[\Tcal_n^m], \hat{\bS}_{\partial\Omega}^{\omega}[\Tcal_n^m]\right) \right).
 \end{split}
\end{equation}
Thus if there exists $n_0$ such that for any $M\in\mathbb{R}_+$
\begin{equation}\label{eq:reson1}
 \Im(\hat{\mu}) |\psi_{1,n_0,m}|^2 >M,
\end{equation}
then resonance occurs. From the expression of $\psi_{1,n,m}$ in \eqref{eq:psi10}, the condition \eqref{eq:reson1} is equivalent to the following one
\begin{equation}\label{eq:reson2}
 \frac{ \Im(\hat{\mu})}{|\widetilde{\psi}_{1,n_0,m}|^2}>M,
\end{equation}
since $f_{1,n_0,m}\neq0$.

Next we perform some asymptotic analysis for the left-hand side of the condition \eqref{eq:reson2} for the large $n_0\gg 1$. From the asymptotic expression of the spherical Bessel and Hankel functions, $j_n(t)$ and $h_n(t)$ in \eqref{eq:asj}, one can obtain that
\begin{equation}\label{eq:coe1}
 \widetilde{\psi}_{1,n_0,m}=C\left( \hat{\mu} + \mu +q_{1,n_0}\right)
\end{equation}
where
\[
 q_{1,n_0}=\Ocal\left( \frac{1}{n_0}\right).
\]
Thus if the parameter $\hat{\mu}$ inside the domain $\Omega$ is chosen as stated in the theorem that
\begin{equation}\label{eq:coe2}
 \hat{\mu}=-\mu+ \rmi/M + p_{1,n_0},
\end{equation}
where
\begin{equation}\label{eq:coe3}
 p_{1,n_0} + q_{1,n_0}=\Ocal(1/M),
\end{equation}
with $q_{1,n_0}$ defined in \eqref{eq:coe1}, then the left-hand side of the condition \eqref{eq:reson2} can be simplified as
\begin{equation}\label{eq:reson3}
 \frac{ \Im(\hat{\mu})}{|\widetilde{\psi}_{1,n_0,m}|^2}\geq M.
\end{equation}
Thus the polariton resonance occurs and the proof is complete.

%\[
%\begin{split}
% E(\bu)= & \Im P_{\hat{\lambda},\hat{\mu}}(\bu_{\delta}, \bu_{\delta}) =\Im \left( -\int_{\Omega} \Lcal_{ \hat{\lambda},\hat{\mu}}\bu_{\delta} \overline{\bu}_{\delta} d\bx + \int_{\partial\Omega} \partial_{\hat{\bnu}}\bu_{\delta}  \overline{\bu}_{\delta} ds \right) \\
% &= \Im \left( \omega^2\int_{\Omega} |\bu_{\delta}|^2 d\bx + \int_{\partial\Omega} \partial_{\hat{\bnu}}\bu_{\delta}  \overline{\bu}_{\delta} ds \right) = \Im \left(  \int_{\partial\Omega} \partial_{\hat{\bnu}}\bu_{\delta}  \overline{\bu}_{\delta}ds \right) \\
% &= \Im \left(  \int_{\partial\Omega}  \left(\hat{\lambda}(\nabla\cdot \bu_{\delta})\bnu + \hat{\mu}(\nabla\bu_{\delta} +\nabla\bu_{\delta}^T)\bnu \right)  \overline{\bu}_{\delta} ds \right) \\
%  &= \Im \left( \hat{\mu} \int_{\partial\Omega}  \left(\hat{\lambda}/\hat{\mu}(\nabla\cdot \bu_{\delta})\bnu +(\nabla\bu_{\delta} +\nabla\bu_{\delta}^T)\bnu \right)  \overline{\bu}_{\delta} ds \right) \\
%   &= \Im \left( \hat{\mu} \int_{\Omega} \Lcal_{ \hat{\lambda}/\hat{\mu},1}\bu_{\delta} \overline{\bu}_{\delta} d\bx  + \hat{\mu} P_{\hat{\lambda},\hat{\mu}}(\bu_{\delta}, \bu_{\delta})  \right)
%\end{split}
%\]
%
\end{proof}
\begin{rem}\label{rem:con}
In Theorem \ref{thm:reson}, we only require the constrain on the Lam\'e parameter $\hat{\mu}$ and there is no restrict on the Lam\'e parameter $\hat{\lambda}$, which indicates that only the first strong convexity condition in \eqref{eq:con} is broken.
\end{rem}
\begin{rem}
We do the numerical simulation to demonstrate that the condition \eqref{con:pla1} can be achieved. The parameters are chosen as follows
 \[
  n_0=5, \ \ \omega=5, \ R=1, \ \ \mu=1,  \ \ \mbox{and} \ \ \Re(\hat{\mu})=-1.87988,
 \]
which is the case beyond the quasi-static approximation from the values of $\omega$ and $R$. The absolute value of the LHS quantity in \eqref{con:pla1} in terms of the parameter $\Im(\hat{\mu})$ is plotted in Fig.~ \ref{fig:reson1}, which evidently demonstrates that the condition \eqref{con:pla1} is fulfilled.
\begin{figure}[t]
  \centering
 {\includegraphics[width=5cm]{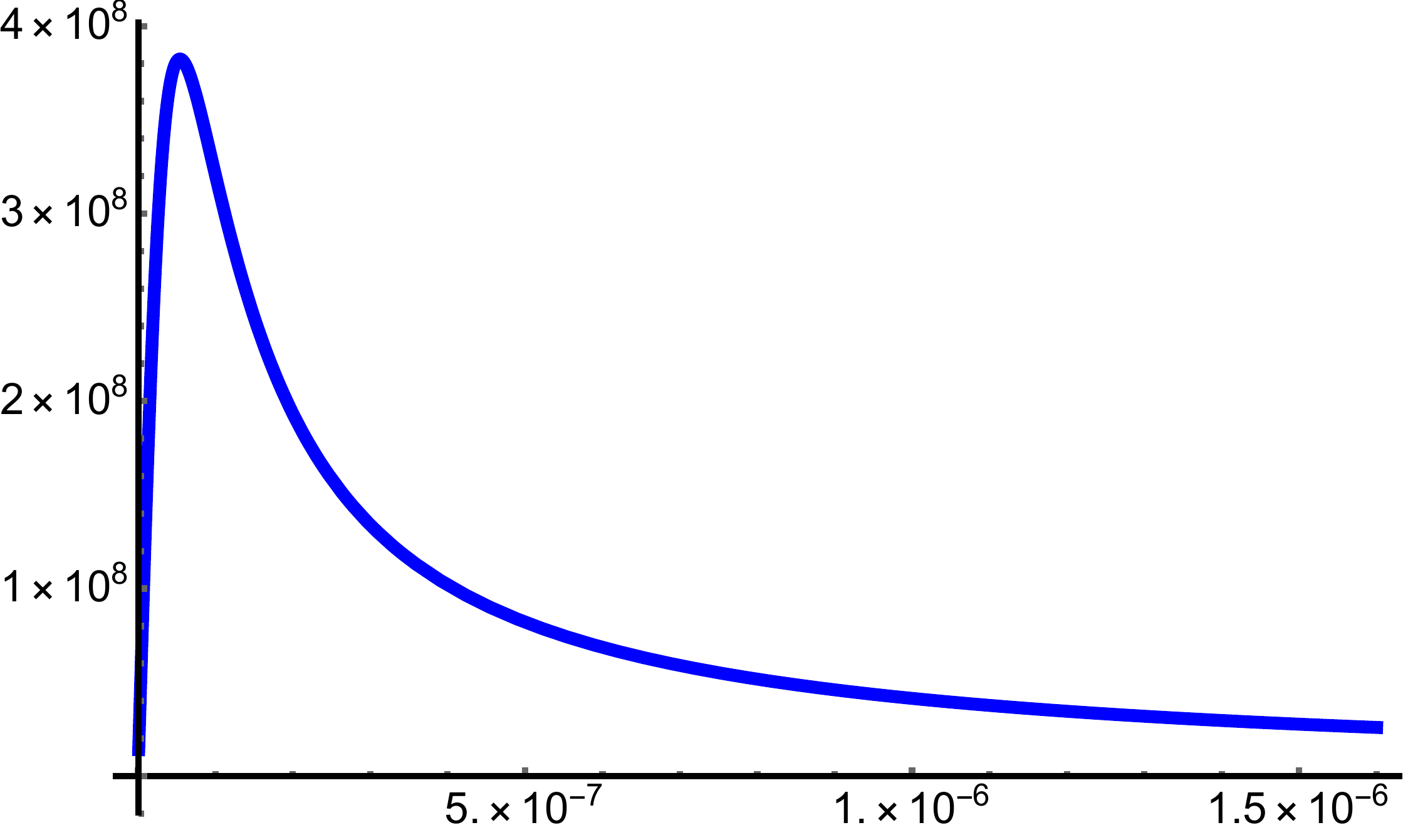}}
  \caption {The absolute value of the LHS quantity in \eqref{con:pla1} in terms of the parameter  $\Im(\hat{\mu})$. }
  \label{fig:reson1}
\end{figure}
\end{rem}

\begin{rem}
Indeed, the condition \eqref{con:n01} is easy to achieve. Since the parameter $q_{1,n_0}$ defined in \eqref{eq:coe1} is of $\Ocal(1/n_0)$, therefore one could choose $p_1=\Ocal(1/n_0)$ to fulfill the condition \eqref{con:n01}.
 Moreover, we do the numerical simulation to demonstrate that the condition  \eqref{con:n01} can be fulfilled. The parameters are chosen as follows
 \[
  n_0=100, \ \ \omega=5, \ R=1, \ \ \mu=1,  \ \  M=10^{10} \ \ \mbox{and} \ \ \hat{\mu}=-\mu+ \rmi/M+ p_{1,n_0}.
 \]
 One can easily check that this is the case beyond quasi-static approximation. The absolute value of the LHS quantity in \eqref{con:n01} in terms of the parameter $p_{1,n_0}$ is plotted in Fig.~ \ref{fig:reson}, which apparently demonstrates that the condition \eqref{con:n01} is satisfied with $p_{1,n_0}\approx 0.02779005=\Ocal(1/n_0)$.
\begin{figure}[t]
  \centering
 {\includegraphics[width=5cm]{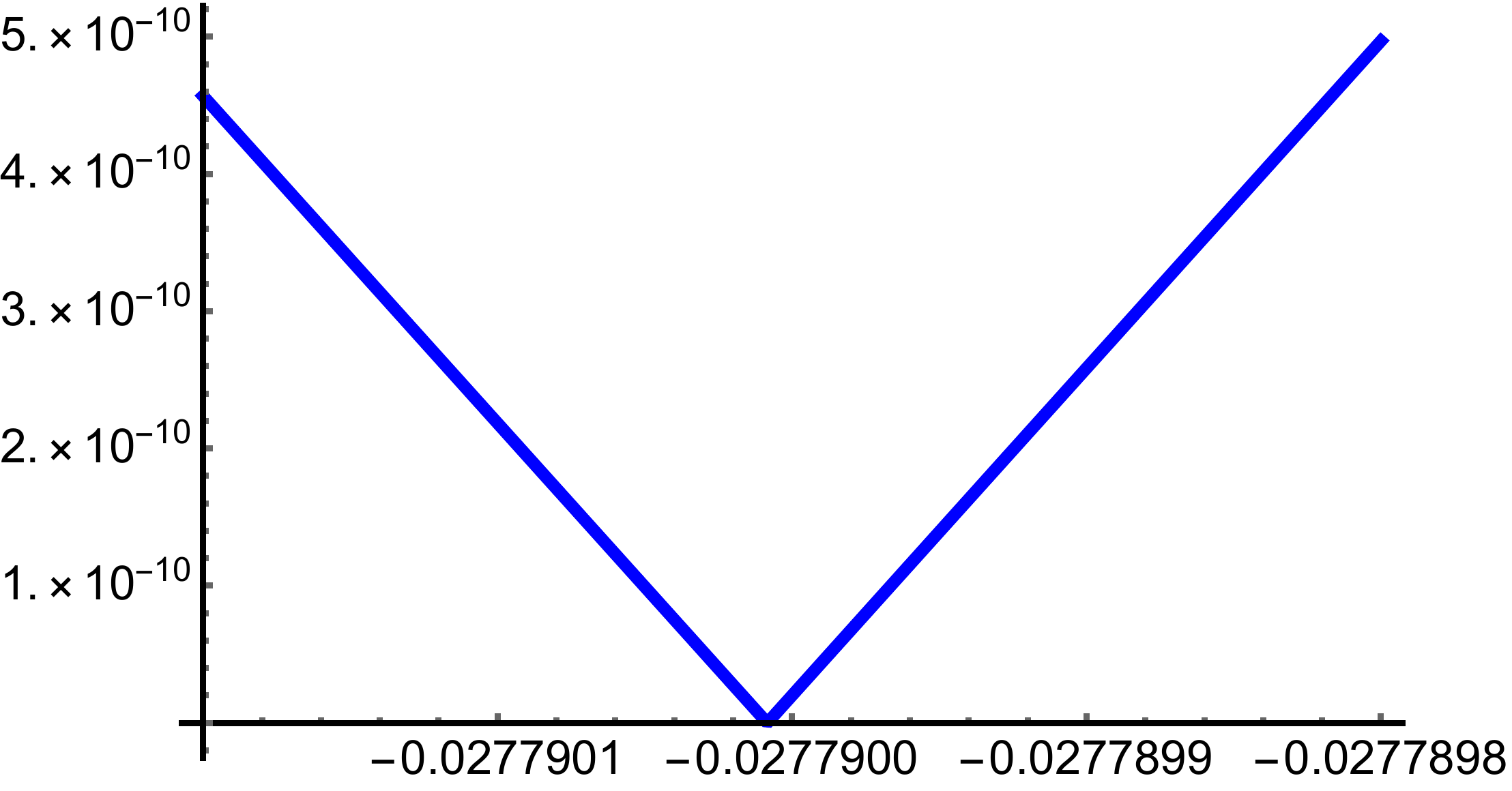}}
  \caption{ The absolute value of the LHS quantity in \eqref{con:n01} in terms of the parameter $p_{1,n_0}$. }
  \label{fig:reson}
\end{figure}
 \end{rem}

 \section{CALR beyond the quasi-static approximation}

In this section, we consider the cloaking effect induced by anomalous localized resonance. In the following, let $D=B_{r_i}$ and $\Omega=B_{r_e}$. To save the notations, we first define the following two functions
\begin{equation}\label{eq:a1}
\begin{split}
\acute{j}_n(t) &= t j^{\prime}_n(t) - j_n(t), \\
\acute{h}_n(t) &= t h^{\prime}_n(t) - h_n(t),
\end{split}
\end{equation}
where $j^{\prime}_n(t)$ and $h^{\prime}_n(t)$ are the derivatives of the functions $j_n(t)$ and $h_n(t)$, respectively. Set 
\[
 \hat{k}_s=\omega/\sqrt{\hat{\mu}}, \quad \mbox{and} \quad \breve{k}_s=\omega/\sqrt{\breve{\mu}},
\]
and we also introduce the following notations,
\begin{equation}\label{eq:a2}
\begin{split}
& j_{n0i}=j_n(k_s r_i) \quad j_{n1i}=j_n(\breve{k}_s r_i), \quad j_{n2i}=j_n(\hat{k}_s r_i),  \\
& j_{n0e}=j_n(k_s r_e) \quad j_{n1e}=j_n(\breve{k}_s r_e), \quad j_{n2e}=j_n(\hat{k}_s r_e),
\end{split}
\end{equation}
and, the same notations hold for the spherical Hankel function $h_n(t)$, the derivative of the  Bessel and Hankel functions, $j^{\prime}_n(t)$ and $h^{\prime}_n(t)$, the functions $\grave{j}_n(t)$ and $\grave{h}_n(t)$ defined in \eqref{eq:asj}, and the functions $\acute{j}_n(t)$ as well as $\acute{h}_n(t)$ defined in \eqref{eq:a1}. Moreover, we let $\mathcal{L}_{\breve{\lambda}, \breve{\mu}}$, $\partial_{\breve{\bnu}}$, $\breve{\bS}_{\partial D}$ and $(\breve{\bK}_{\partial D}^{\omega})^* $, respectively, denote the Lam\'e operator, the associated conormal derivative, the single layer potential operator and the N-P operator associated with the Lam\'e parameters $(\breve{\lambda}, \breve{\mu})$.

 Assume that the source $\bff\in H^{-1}(\mathbb{R}^3)^3$ is compactly supported outside $\Omega$, then the elastic system \eqref{eq:lame1} can be expressed as the following equation system
\begin{equation}\label{eq:calr1}
  \left\{
    \begin{array}{ll}
      \mathcal{L}_{\breve{\lambda},\breve{\mu}}\bu(\bx) + \omega^2\bu(\bx)  =0 , & \mbox{in} \ \ D , \medskip  \\
      \mathcal{L}_{\hat{\lambda},\hat\mu}\bu(\bx)+ \omega^2\bu(\bx)  =0 , & \mbox{in} \ \ \Omega\backslash \overline{D} ,\medskip   \\
      \mathcal{L}_{\lambda, \mu}\bu(\bx) + \omega^2\bu(\bx)   =\bff, &  \mbox{in} \ \ \mathbb{R}^3\backslash \overline{\Omega},\medskip \\
      \bu|_- = \bu|_+, \quad \partial_{\breve{\bnu}}\bu|_- = \partial_{\hat{\bnu}}\bu|_+   & \mbox{on} \ \ \partial D ,\medskip\\
      \bu|_- = \bu|_+, \quad  \partial_{\hat{\bnu}}\bu|_- = \partial_{\bnu}\bu|_+ & \mbox{on} \; \partial \Omega.
    \end{array}
  \right.
\end{equation}
With the help of the potential theory, the solution to the equation system \eqref{eq:calr1} can be represented by
\begin{equation}\label{eq:sc1}
  \bu(\bx)=
  \left\{
    \begin{array}{ll}
      \breve{\bS}^\omega_{\partial D}[\bvarphi_1](\bx), & \bx\in D,\medskip \\
      \hat{\bS}^\omega_{\partial D}[\bvarphi_2](\bx) + \hat{\bS}^\omega_{\partial\Omega}[\bvarphi_3](\bx), & \bx\in \Omega\backslash \overline{D},\medskip \\
      \bS^\omega_{\partial\Omega}[\bvarphi_4](\bx) + \mathbf{F}(\bx), & \bx\in \mathbb{R}^3\backslash \overline{\Omega},
    \end{array}
  \right.
\end{equation}
where $\bvarphi_1, \bvarphi_2\in L^2(\partial D)^3$, $\bvarphi_3, \bvarphi_4\in L^2(\partial\Omega)^3$ and $\mathbf{F}$ is the Newtonian potential of the source $\bff$ defined in \eqref{eq:newpf}. One can easily see that the solution given \eqref{eq:sc1} satisfies the first three condition in \eqref{eq:calr1} and the last two conditions on the boundary yield that
\begin{equation}\label{eq:sc2}
  \left\{
    \begin{array}{ll}
      \breve{\bS}^\omega_{\partial D}[\bvarphi_1]=\hat{\bS}^\omega_{\partial D}[\bvarphi_2] + \hat{\bS}^\omega_{\partial \Omega}[\bvarphi_3], & \mbox{on} \quad \partial D,\medskip \\
       \partial_{\breve{\bnu}}\breve{\bS}^\omega_{\partial D}[\bvarphi_1|_- = \partial_{\hat{\bnu}}(\hat{\bS}^\omega_{\partial D}[\bvarphi_2] + \hat{\bS}^\omega_{\partial \Omega}[\bvarphi_3])|_+ , & \mbox{on} \quad \partial D, \medskip \\
      \hat{\bS}^\omega_{\partial D}[\bvarphi_2] + \hat{\bS}^\omega_{\partial \Omega}[\bvarphi_3]= \bS^\omega_{\partial \Omega}[\bvarphi_4] + \mathbf{F}, & \mbox{on} \quad \partial\Omega, \medskip \\
      \partial_{\hat{\bnu}}(\hat{\bS}^\omega_{\partial D}[\bvarphi_2] + \hat{\bS}^\omega_{\partial \Omega}[\bvarphi_3])|_- = \partial_{\bnu}(\bS^\omega_{\partial \Omega}[\bvarphi_4] + \mathbf{F})|_+ , & \mbox{on} \quad \partial \Omega.
    \end{array}
  \right.
\end{equation}
With the help of the jump formual in \eqref{eq:jump}, the equation system \eqref{eq:sc2} further yields the following integral system,
\begin{equation}\label{eq:cma}
  \left[
    \begin{array}{cccc}
        \breve{\bS}^\omega_{\partial D} & -\hat{\bS}^\omega_{\partial D} & -\hat{\bS}^\omega_{\partial \Omega} & 0 \\
      -\frac{1}{2} +(\breve{\bK}_{\partial D}^{\omega})^*   & -\frac{1}{2} - (\hat{\bK}_{\partial\Omega}^{\omega})^*  &\partial_{\hat{\bnu}_i} \hat{\bS}^\omega_{\partial\Omega} & 0 \\
      0 &  \hat{\bS}^\omega_{\partial D} & \hat{\bS}^\omega_{\partial \Omega} & - \bS^\omega_{\partial \Omega} \\
      0 & \partial_{\hat\bnu_e}\hat{\bS}^\omega_{\partial D} & -\frac{1}{2} + (\hat{\bK}_{\partial\Omega}^{\omega})^* & -\frac{1}{2} -({\bK}_{\partial\Omega}^{\omega})^* \\
    \end{array}
  \right]
\left[
  \begin{array}{c}
    \bvarphi_1 \\
    \bvarphi_2 \\
    \bvarphi_3 \\
    \bvarphi_4 \\
  \end{array}
\right]=
\left[
  \begin{array}{c}
    0 \\
    0 \\
    \mathbf{F} \\
    \partial_{\bnu}\mathbf{F} \\
  \end{array}
\right],
\end{equation}
where $\partial_{\hat{\bnu}_i}$ and $ \partial_{\hat\bnu_e}$ signify the conormal derivatives on the boundaries of $D$ and $\Omega$, respectively.

In the following, we assume that the Newtonian potential $\bF$ of the source $\bff$ has the following expression
\begin{equation}\label{eq:FF2}
  \bF= \sum_{n=N}^{\infty} \sum_{m=-n}^{n} \left(f_{1,n,m} j_n(k_s|\bx|) \Tcal_n^m \right) \quad\mbox{for}\quad \bx\in\Omega,
\end{equation}
where $N$ is large enough such the the spherical Bessel and Hankel functions, $j_n(t)$ and $h_n(t)$, fulfill the asymptotic expansions shown in \eqref{eq:asj}. From the Theorem \ref{thm:singleo} and the orthogonality of the functions $\Tcal_n^m$, $\Ical_{n-1}^m$ and $ \Ncal_{n+1}^m$, one can deduce that the density functions $ \bvarphi_i$, $i=1,2,3,4$ can be written as follows
\begin{equation}\label{eq:co1}
\begin{split}
\bvarphi_1&=\sum_{n=N}^{+\infty}\sum_{m=-n}^n \varphi_{1,n,m} \Tcal_n^m,  \qquad  \bvarphi_2=\sum_{n=N}^{+\infty}\sum_{m=-n}^n \varphi_{2,n,m} \Tcal_n^m,\\
\bvarphi_3&=\sum_{n=N}^{+\infty}\sum_{m=-n}^n \varphi_{3,n,m} \Tcal_n^m,  \qquad  \bvarphi_4=\sum_{n=N}^{+\infty}\sum_{m=-n}^n \varphi_{4,n,m} \Tcal_n^m.
\end{split}	
\end{equation}
With the help of the equation \eqref{eq:sint} as well as the Theorem \ref{thm:ks} and by substituting the expressions in \eqref{eq:FF2} and \eqref{eq:co1}  into the equation system \eqref{eq:cma}, the integral system can be reduced the following  equation system
\begin{equation}\label{eq:cma1}
  \left[
    \begin{array}{cccc}
        a_{11} & a_{12} & a_{13} & 0 \\
      a_{21}  & a_{22}  &a_{23} & 0 \\
      0 &  a_{32} & a_{33} & a_{34} \\
      0 &  a_{42} & a_{43} & a_{44} \\
    \end{array}
  \right]
\left[
  \begin{array}{c}
    \varphi_{1,n,m} \\
    \varphi_{2,n,m} \\
    \varphi_{3,n,m} \\
    \varphi_{4,n,m} \\
  \end{array}
\right]=
\left[
  \begin{array}{c}
    0 \\
    0 \\
    f_{1,n,m} j_{n0e} \\
    g_{1,n,m} \\
  \end{array}
\right],
\end{equation}
where
\[
 a_{11} =\frac{-\rmi \breve{k}_s r_i^2 j_{n1i} h_{n1i} }{ \breve{\mu} }, \quad a_{12} =\frac{-\rmi \hat{k}_s r_i^2 j_{n1i} h_{n2i}}{ \hat{\mu} }, \quad a_{13} =\frac{-\rmi \hat{k}_s r_e^2 j_{n1i} h_{n2e}}{ \hat{\mu} },
\]
\[
 a_{21} =-\rmi \breve{k}_s r_i  \acute{j}_{n1i} h_{n1i} , \quad a_{22} =-\rmi \hat{k}_s r_i  j_{n2i} \acute{h}_{n2i}, \quad a_{23} =-\rmi \hat{k}_s r_e  \acute{j}_{n2i} h_{n2e},
 \]
\[
 a_{32} =\frac{-\rmi \hat{k}_s r_i^2 j_{n2i} h_{n2e} }{ \hat{\mu} }, \quad a_{33} =\frac{-\rmi \hat{k}_s r_e^2 j_{n2e} h_{n2e}}{ \hat{\mu} }, \quad a_{34} =\frac{\rmi k_s r_e^2 j_{n0e} h_{n0e}}{\mu },
\]
\[
 a_{42} =-\rmi \hat{k}_s r_i  j_{n2i} \acute{h}_{n2e} , \quad a_{43} =-\rmi \hat{k}_s r_e  \acute{j}_{n2e}  h_{n2e}, \quad a_{44} =\rmi k_s r_e  j_{n0e} \acute{h}_{n0e},
 \]
 and
 \[
  g_{1,n,m}= f_{1,n,m}\mu  \left(k_s r_e  j_{n0e}^{\prime}- j_{n0e} \right)/r_e.
 \]
Solving the equation system \eqref{eq:cma1} gives that
\begin{equation}\label{eq:co2}
\begin{split}
\varphi_{1,nm}&=\frac{ \widetilde{\varphi}_{1,n,m} }{d_{n,m}},  \qquad  \varphi_{2,nm}=\frac{ \widetilde{\varphi}_{2,n,m} }{d_{n,m}},\\
\varphi_{3,nm}&=\frac{ \widetilde{\varphi}_{3,n,m} }{d_{n,m}},  \qquad  \varphi_{4,nm}=\frac{ \widetilde{\varphi}_{4,n,m} }{d_{n,m}},
\end{split}	
\end{equation}
where
\[
 \widetilde{\varphi}_{1,n,m} = \frac{- \rmi k \hat{k}^2 r_i r_e^2 f_{1,n,m} h_{n23}j_{n0e} j_{n2i} (\acute{h}_{n0e} j_{n0e} - \acute{j}_{n0e}h_{n0e} ) (\acute{h}_{n2i} j_{n2i} r_e - \acute{j}_{n2i}h_{n2i} r_i ) }{\hat{\mu}},
\]
\[
 \widetilde{\varphi}_{2,n,m} = \frac{ \rmi k \breve{k} \hat{k} r_i r_e^2 f_{1,n,m} h_{n1i} h_{n2e} j_{n0e} (\acute{h}_{n0e} j_{n0e} - \acute{j}_{n0e}h_{n0e} ) (\acute{j}_{n2i} j_{n1i} \hat{\mu} r_i - \acute{j}_{n1i}j_{n2i} \breve{\mu} r_e ) }{\breve{\mu}\hat{\mu}},
\]
\[
 \widetilde{\varphi}_{3,n,m} = \frac{ \rmi k \breve{k} \hat{k} r_i^3 r_e f_{1,n,m} h_{n1i} j_{n2i} j_{n0e} (\acute{h}_{n0e} j_{n0e} - \acute{j}_{n0e}h_{n0e} ) ( \acute{j}_{n1i}h_{n2i} \breve{\mu} - \acute{h}_{n2i} j_{n1i} \hat{\mu}  ) }{\breve{\mu}\hat{\mu}},
\]

\[
\begin{split}
 \widetilde{\varphi}_{4,n,m} = & \frac{ \rmi  \breve{k} \hat{k}^2 r_i^2 f_{1,n,m} \left(  \acute{j}_{n0e} \mu r_i \left(  j_{n1i} (\acute{h}_{n2i}  j_{n2e} r_e-\acute{j}_{n2i} h_{n2e}  r_i) \hat{\mu}  +  \acute{j}_{n1i} (h_{n2e} j_{n2i} - h_{n2i} j_{n2e}) \breve{\mu} r_e \right) \right) }{\breve{\mu}\hat{\mu}^2}\times \\
 &  \frac{ h_{n1i} h_{n2e} j_{n2i}  \left(  j_{n0e} \hat{\mu} r_e \left( \acute{j}_{n1i} (\acute{j}_{n2e}  h_{n2i} r_i - \acute{h}_{n2e} j_{n2i}  r_e) \breve{\mu}  +  j_{n1i} (\acute{h}_{n2e} \acute{j}_{n2i} - \acute{h}_{n2i} \acute{j}_{n2e}) \hat{\mu} r_i \right) \right) }{\breve{\mu}\hat{\mu}^2},
 \end{split}
\]
and

\[
\begin{split}
d_{n,m} = & \frac{k \breve{k} \hat{k}^2 r_i^2 r_e^2 \left(  \acute{h}_{n0e} \mu r_i \left(  j_{n1i} (\acute{h}_{n2i}  j_{n2e} r_e+\acute{j}_{n2i} h_{n2e}  r_i) \hat{\mu}  -  \acute{j}_{n1i} (h_{n2e} j_{n2i} - h_{n2i} j_{n2e}) \breve{\mu} r_e \right) \right) }{\mu \breve{\mu}\hat{\mu}^2}\times \\
 &  \frac{ h_{n1i} h_{n2e} j_{n0e} j_{n2i}  \left(  h_{n0e} \hat{\mu} r_e \left( \acute{j}_{n1i} (  \acute{h}_{n2e} j_{n2i}  r_e - \acute{j}_{n2e}  h_{n2i} r_i ) \breve{\mu}  +  j_{n1i} (\acute{h}_{n2i} \acute{j}_{n2e} - \acute{h}_{n2e} \acute{j}_{n2i}) \hat{\mu} r_i \right) \right) }{\mu \breve{\mu}\hat{\mu}^2}.
 \end{split}
\]
To simplify the exposition, we introduce the following two notations
\begin{equation}
\eta_{n2e}=n-1+n \grave{j}^{\prime}_{n2e} -\grave{j}_{n2e},
\end{equation}
and
\begin{equation}
\gamma_{n2e}=n+2+ (n+1) \grave{h}^{\prime}_{n2e} + \grave{h}_{n2e},
\end{equation}
where $\grave{j}^{\prime}_{n2e}$, $\grave{j}_{n2e}$, $\grave{h}^{\prime}_{n2e}$ and $\grave{h}_{n2e}$ are defined in \eqref{eq:a2}. The same notations also hod for $\eta_{n1i}$, $\eta_{n21}$, $\gamma_{n0e}$ and $\gamma_{n2i}$. We also define the following function
\begin{equation}\label{pa:q2}
\begin{split}
q_{2,n}(\breve{\mu}, \hat{\mu}, r_i, r_e) = & (\breve{\mu} +\hat{\mu} ) (\mu + \hat{\mu})n^2 r_e^2 +  ( \hat{\mu}r_i - \breve{\mu}r_e )(\mu r_i -\hat{\mu}r_e) n^2\rho^{2n} -       \\
& \hat{\mu}r_e (1+\grave{h}_{n0e}) \bigg(  \breve{\mu} r_e \eta_{n11} \Big( \gamma_{n2e} \rho^{2n}(1+\grave{j}_{n2i}) + (1+\grave{h}_{n2i})\eta_{n2e} \Big)-  \\
  & \hat{\mu} (1+\grave{j}_{n1i})\Big(  r_i \gamma_{n2e}\eta_{n2i}\rho^{2n} - r_e\gamma_{n2i}\eta_{n2e}  \Big)     \bigg) -\\
  &  \mu \gamma_{n0e} \bigg( \breve{\mu} r_e \eta_{n1i} \Big( r_e(1+\grave{h}_{n2i})(1+\grave{j}_{n2e}) -r_i\rho^{2n}(1+\grave{h}_{n2e})(1+\grave{j}_{n2i})   \Big)  \\
  &  \hat{\mu} (1+\grave{j}_{n1i})\Big( r_i^2 \rho^{2n} (1+\grave{h}_{n2e})\eta_{n2i} + r_e^2  (1+\grave{j}_{n2e})\gamma_{n2i}  \Big)  \bigg),
\end{split}
\end{equation}
here and also in what follows, $\rho=r_i/r_e$.

With the above preparation, we are in a position to show the CALR result, which is concluded in the following theorem.
\begin{thm}\label{thm:CALR}
Consider the configuration $(\mathbf{C}_0, \bff)$ where $\mathbf{C}_0$ is given in \eqref{eq:pa1}. Suppose that the Newtonian potential $\bF$ of the source term $\bff$ has the expression shown in \eqref{eq:FF2} with $f_{1,n_0,m}\neq0$ for some $n_0\in\mathbb{N}$.  For any $M\in\mathbb{R}_+$, if the parameters in $\mathbf{C}_0$ are chosen as follows
\begin{equation}\label{con:002}
\breve{\mu}=\mu, \quad \mbox{and} \quad \hat{\mu}=-\mu + \rmi\rho^{n_0} + p_{2,n_0}, 
\end{equation}	
such that
\begin{equation}\label{eq:con1}
 p_{2,n_0}^2 +q_{2,n_0}=\Ocal \left( \rho^{2n_0} \right) 
\end{equation}
and
\begin{equation}\label{eq:con2}
n_0 \left(1 + \tau_1  \frac{k^2r_e^3}{r_i}\right)^{n_0}>M,
\end{equation}
where $q_{2,n_0}$ is defined in \eqref{pa:q2} and $\tau_1\in\mathbb{R}_+$ is given in \eqref{eq:lll2},  then the phenomenon of the CALR could occur if the source supported inside the critical radius $r_*=\sqrt{r_e^3/r_i}$. Moreover, if the source is supported outside $B_{r_*}$, then there is no resonant result.
\end{thm}

\begin{proof}
We first show the polariton resonance, namely the condition \eqref{con:res}. For notational convenience of the proof, we set
\[
 \tilde{f}_{1,n,m}:=\frac{f_{1,n,m}}{(2n+1)!!}, \quad n\geq N.
\]
 When $N$ is large enough such that the spherical Bessel and Hankel functions, $j_n(t)$ and $h_n^{(1)}(t)$, enjoy the asymptotic expression shown in \eqref{eq:asj}, direct calculations show that the coefficients satisfy the following estimates
\begin{equation}
 |\widetilde{\varphi}_{2,n,m}|  \approx \frac{f_{1,n,m} (\hat{k}_s r_i )^n}{(2n+1)!!}, \qquad  |\widetilde{\varphi}_{3,n,m}|\approx\frac{f_{1,n,m} \rho^{n_0}  (\hat{k}_sr_e)^n }{(2n+1)!!},
\end{equation}
\begin{equation}\label{eq:e4n}
| \widetilde{\varphi}_{4,n,m}|\leq \frac{f_{1,n,m} (kr_e)^{n}}{(2n+1)!!}.
\end{equation}
Moreover, the condition \eqref{eq:con1} yields that when $n=n_0$,
\begin{equation}\label{eq:coed1}
|d_{n_0,m}| \approx  \rho^{2n_0},
\end{equation}
and when $n\neq n_0$,
\begin{equation}\label{eq:coed2}
|d_{n,m}| \geq \rho^{2n_0} + \rho^{2n}.
\end{equation}
Thus from \eqref{eq:sc1}, the displacement field $\bu$ to the system \eqref{eq:calr1} in the shell $\Omega\backslash\overline{D}$ can be represented as
\begin{equation}
\begin{split}
\bu&=    \hat{\bS}^\omega_{\partial D}[\bvarphi_2](\bx) + \hat{\bS}^\omega_{\partial\Omega}[\bvarphi_3](\bx) \\
                  &=   \sum_{n=N}^{\infty}\sum_{m=-n}^n -\frac{i \hat{k}_s}{\hat{\mu}} \Big( \varphi_{2,n,m} r_i^2 j_{n2i}h_n(\hat{k}_s|\bx|) + \varphi_{3,n,m} r_e^2 h_{n2e}j_n(\hat{k}_s|\bx|)  \Big)  \Tcal_n^m,
\end{split}
\end{equation}
where $\varphi_{2,n,m}$ and $\varphi_{3,n,m}$ are defined in \eqref{eq:co2}.

Next we give the estimate of the dissipation energy $E(\bu)$.  From the definition of the dissipation energy $E(\bu)$ in \eqref{def:E} and with the help of Green's formula, one can have the following estimate
\begin{equation}\label{eq:E2}
\begin{split}
 E(\bu) &=  \Im P_{\hat{\lambda},\hat{\mu}}(\bu, \bu) =  \Im\left(  \int_{\partial\Omega} \partial_{\hat{\bnu}} \bu \overline{\bu}ds  -   \int_{\partial D} \partial_{\hat{\bnu}} \bu \overline{\bu}ds \right) \geq \tilde{f}_{1,n_0,m}^2 \left( \frac{k^2r_e^3}{r_i} \right)^{n_0}
  \end{split}
\end{equation}
If the source $\bff$ is supported inside the critical radius $r_*=\sqrt{r_e^3/r_i}$, by \eqref{eq:FF2} and the asymptotic property of $j_n(t)$ in \eqref{eq:asj}, one can verify that there exists $\tau_1\in\mathbb{R}_+$ such that
\begin{equation}\label{eq:lll2}
 \limsup_{n\rightarrow\infty} (\tilde{f}_{1,n,m})^{1/n}=\sqrt{\frac{r_i}{k^2r_e^3}+\tau_1}.
\end{equation}
Combining  \eqref{eq:E2} as well as \eqref{eq:lll2} and together with the help of condition \eqref{eq:con2}, one can obtain that
\[
 E(\bu) \geq n_0 \left(\frac{r_i}{k^2r_e^3}+\tau_1\right)^{n_0} \left( \frac{k^2r_e^3}{r_i} \right)^{n_0}>M,
\]
which exactly shows that the polariton resonance occurs, namely the condition \eqref{con:res} is fulfilled.

Then we consider the case when the source is supported outside the critical radius $r_*$. Thus there exists $\tau_2>0$ such that
\[
 \limsup_{n\rightarrow\infty} (\tilde{f}_{1,n,m})^{1/n}\leq\frac{1}{kr_*+\tau_2},
\]
and the dissipation energy $E(\bu)$ can be estimated as follows
\[
 \begin{split}
   E(\bu) & \leq \sum_{n\geq N}  \frac{\tilde{f}_{1,n,m}^2 (kr_e)^{2n} \rho^{n_0} } {\rho^{2n_0}+\rho^{2n}} \leq \sum_{n\geq N} \tilde{f}_{1,n,m}^2 \left( \frac{k^2r_e^3}{r_i} \right)^{n_0}\leq C,
 \end{split}
\]
which means that resonance does not occur.

Next we prove the boundedness of the solution $\bu$ when $|x|>r_e^3/r_i^2$. From \eqref{eq:sc1}, \eqref{eq:co1} and \eqref{eq:co2}, the displacement field $\bu$ in $\mathbb{R}^3\backslash \overline{\Omega}$  can be represented as
\begin{equation}
\begin{split}
\bu=   \sum_{n=N}^{\infty}\sum_{m=-n}^n -\frac{\rmi k_s}{ \mu} \Big( \varphi_{4,n,m} r_e^2 j_{n0e}h_n(\hat{k}_s|\bx|)   \Big)  \Tcal_n^m + \bF(\bx),
\end{split}
\end{equation}
Moreover, from \eqref{eq:e4n}, \eqref{eq:coed1} and \eqref{eq:coed2}, one can obtain that
\begin{equation}\label{bound1}
 \begin{split}
   |\bu|  \leq  \sum_{n=N}^{\infty}\sum_{m=-n}^n |\tilde{f}_{1,n,m}|(kr_e)^{n_0} \left(\frac{r_e^3}{r_i^2}\right)^{n}\frac{1}{r^{n}} + |\bF|\leq C,
 \end{split}
\end{equation}
when $|x|>r_e^3/r_i^2$.

This completes the proof.
\end{proof}
\begin{rem}
Similar to Remark \ref{rem:con}, in Theorem \ref{thm:CALR}, we only require the constrain on the Lam\'e parameter $\hat{\mu}$ and there is no restrict on the Lam\'e parameter $\hat{\lambda}$, which indicates that only the first strong convexity condition in \eqref{eq:con} is broken.
\end{rem}
\begin{rem}\label{rem:genf}
In Theorem \ref{thm:CALR},the constrain on the source $\bff$, whose Newtonian potential $\bF$ should have the expression in \eqref{eq:FF2}, is just a technical issue. Indeed, the phenomenon of the CALR could occur for a general source term $\bff$. The reason we require $N$ in \eqref{eq:FF2} should be large is that we need to apply the asymptotic properties of the spherical Bessel and Hankel functions, $j_n(t)$ and $h_n(t)$ to prove the polariton resonance condition \eqref{con:res} and the boundedness condition \eqref{con:bou}. However, for the condition \eqref{con:res}, the ALR is a spectral phenomenon at the limit point of eigenvalues of the N-P operator, which naturally requires that the order $n$ should be large. While for the condition \eqref{con:bou}, if the item possessing the polariton resonance is bounded, then the other items are spontaneously bounded outside a certain region. Therefore the CALR could occur for a general source term $\bff$.
\end{rem}

\begin{rem}
We do the numerical simulation to show that the condition  \eqref{eq:con1} can be fulfilled. The parameters are chosen as follows
 \[
  n_0=50, \ \ \omega=5, \ \  r_i=0.8, \ \  r_e=1, \ \  \breve{\mu}=\mu=1 \ \  \mbox{and} \ \  (r_i/r_e)^{2n_0}\approx 2\times 10^{-10},
 \]
From the values of the parameters $\omega$ and $r_e$, one can readily verify that this is the case beyond quasi-static approximation. The norm of the LHS quantity in \eqref{eq:con1} in terms of the parameter $p_{2,n_0}$ is plotted in Fig.~ \ref{fig:CALR}, which apparently demonstrates that the condition \eqref{eq:con1} is satisfied.
\begin{figure}[t]
  \centering
 {\includegraphics[width=5cm]{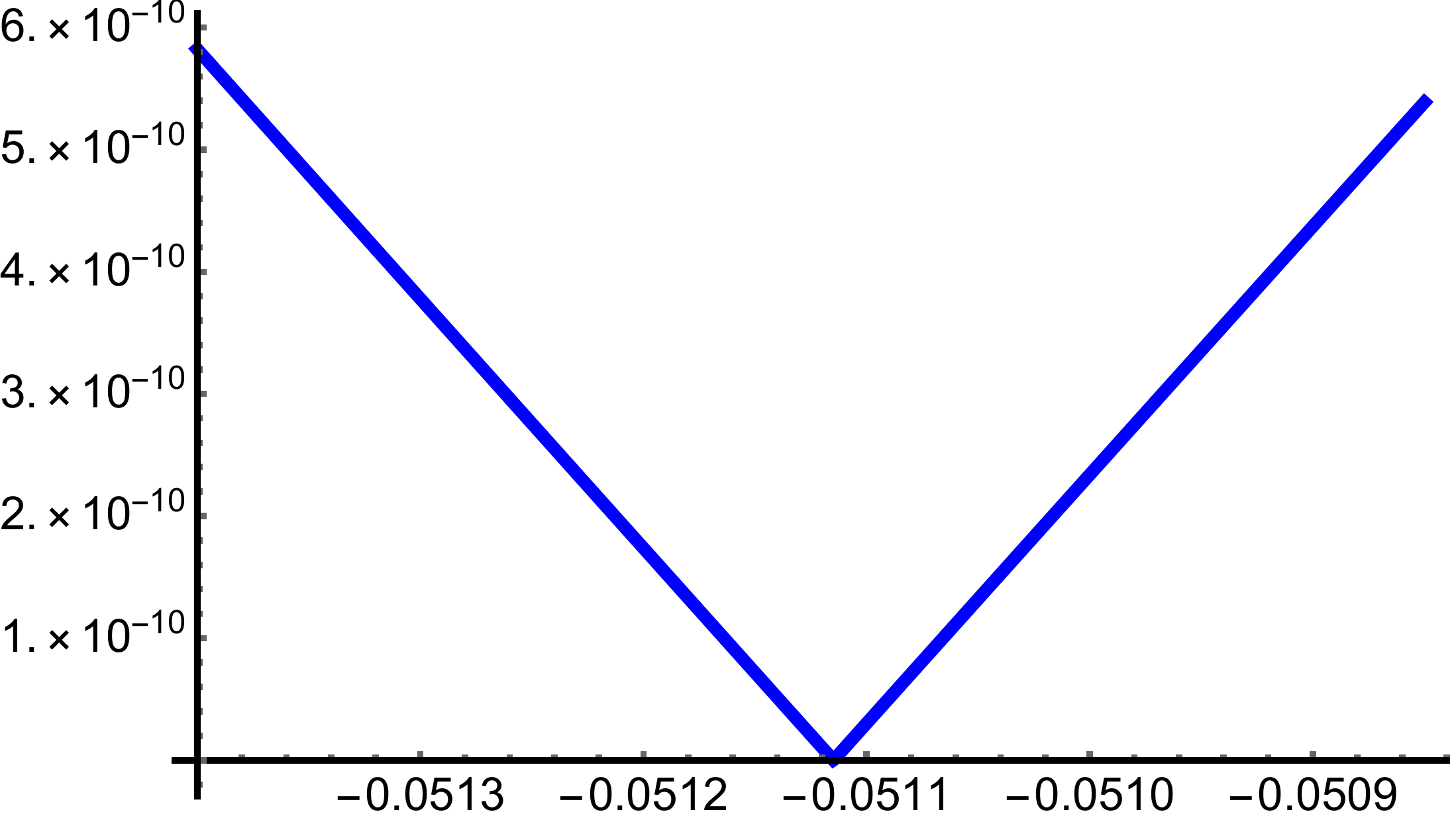}}
  \caption{The absolute value of the LHS quantity in \eqref{eq:con1} with respect the change of the parameter $p_2$. }
  \label{fig:CALR}
\end{figure}

\end{rem}


\begin{thebibliography}{99}
%\bibitem{HK07:book} {H.~Ammari, H.~Kang}, {\it Polarization and Moment Tensors
%With Applications to Inverse Problems and Effective Medium Theory}, Applied Mathematical Sciences, Springer-Verlag, Berlin Heidelberg, 2007.

\bibitem{Ack14}
H.~Ammari, G.~Ciraolo, H.~Kang, H.~Lee, and G.W. Milton, \emph{Spectral theory of a Neumann-Poincar\'{e}-type operator and analysis of cloaking due to anomalous localized resonance {I}{I}},
  Contemporary Math., \textbf{615} (2014), 1--14.

\bibitem{Acm13}
H.~Ammari, G.~Ciraolo, H.~Kang, H.~Lee, and G.W. Milton, \emph{Anomalous
  localized resonance using a folded geometry in three dimensions,}, Proc. R.
  Soc. A, \textbf{469} (2013), 20130048.

\bibitem{Ack13}
H.~Ammari, G.~Ciraolo, H.~Kang, H.~Lee, and G.W. Milton, \emph{Spectral theory of a Neumann-Poincar\'{e}-type operator and
  analysis of cloaking due to anomalous localized resonance}, Arch. Ration.
  Mech. Anal., \textbf{208} (2013), 667--692.

\bibitem{ADM} {H.~Ammari, Y. Deng and P. Millien}, \emph{Surface plasmon resonance of nanoparticles and applications in imaging},
Arch. Ration. Mech. Anal., \textbf{220} (2016), 109--153.


\bibitem{AMRZ} {H.~Ammari, P. Millien, M. Ruiz and H. Zhang}, \emph{Mathematical analysis of plasmonic nanoparticles: the scalar case}, Arch. Ration. Mech. Anal., DOI: 10.1007/s00205-017-1084-5

\bibitem{ARYZ} {H.~Ammari, M. Ruiz, S. Yu and H. Zhang}, \emph{Mathematical analysis of plasmonic resonances for nanoparticles: the full Maxwell equations}, J. Differential Equations, {\bf 261} (2016), 3615--3669.

\bibitem{AK} K. Ando and H. Kang, {\it Analysis of plasmon resonance on smooth domains using spectral properties of the Neumann-Poincar\'e operator}, J. Math. Anal. Appl., \textbf{435} (2016), 162--178.

\bibitem{AJKKY15} K. Ando, Y. Ji, H. Kang, K. Kim and S. Yu, \emph{Spectral properties of the Neumann-Poincar\'e operator and cloaking by anomalous localized resonance for the elasto-static system}, arXiv:1510.00989.

\bibitem{AKKY16} K. Ando, H. Kang, K. Kim and S. Yu, \emph{Cloaking by anomalous localized resonance for linear elasticity on a coated structure}, arXiv:1612.08384.

\bibitem{AKL} {K. Ando, H. Kang and H. Liu}, \emph{Plasmon resonance with finite frequencies: a validation of the quasi-static approximation for diametrically small inclusions}, SIAM J. Appl. Math., \textbf{76} (2016), 731--749.

\bibitem{BL}
{O. Bruno and S. Lintner}, Superlens-cloaking of small dielectric bodies in the quasistatic regime, Journal of Applied Physics \textbf{102} (2007), no. 12.

\bibitem{BLL}
 {E. Bl{\aa}sten, H. Li, H. Liu and Y. Wang}, Localization and geometrization in plasmon resonances and geometric structures of Neumann-Poincar\'e eigenfunctions, arXiv:1809.08533.


\bibitem{Bos10}
G.~Bouchitt\'{e} and B.~Schweizer, \emph{Cloaking of small objects by anomalous localized resonance}, Quart. J. Mech. Appl. Math., \textbf{63} (2010), 438--463.

\bibitem{Brl07} O.P. Bruno and S.~Lintner, \emph{Superlens-cloaking of small dielectric bodies in the quasistatic regime}, J. Appl. Phys., \textbf{102} (2007), 124502.

% \bibitem{CKKL} D. Chung, H. Kang, K. Kim and H. Lee, {\it Cloaking due to anomalous localized resonance in plasmonic structures of confocal ellipses}, preprint, arXiv: 1306.6679.
%
 \bibitem{CK} {D.~Colton and R.~Kress}, {\it Inverse Acoustic and Electromagnetic Scattering Theory}, 2nd Edition, Springer-Verlag, Berlin, 1998.

 \bibitem{DLL} Y. Deng, H. Li and H. Liu, {\it On spectral properties of Neumann-Poincare operator and plasmonic cloaking in 3D elastostatics}, \textsf{J. Spectral Theory}, DOI:10.4171/JST/262.

\bibitem{DLL1} {Y. Deng, H. Li and H. Liu}, {Analysis of surface polariton resonance for nanoparticles in elastic system},  arXiv:1804.05480.

\bibitem{KLO} H. Kettunen, M. Lassas and P. Ola, {\it On absence and existence of the anomalous localized resonace without the quasi-static approximation}, preprint, arXiv: 1406.6224.

%\bibitem{KM} {D. M. Kochmann and G. W. Milton}, \emph{Rigorous bounds on the effective moduli of
%composites and inhomogeneous bodies with negative-stiffness phases}, J. Mech. Phys.
%Solids, {\bf 71} (2014), 46--63.
%
\bibitem{Klsap} {R.V. Kohn, J.Lu, B.~Schweizer and M.I. Weinstein}, \emph{A variational
  perspective on cloaking by anomalous localized resonance}, Comm. Math. Phys., {\bf 328} (2014), 1--27.

  \bibitem{Kup} V. D. Kupradze, {\it Three-dimensional Problems of the Mathematical Theory of Elasticity and Thermoelasticity}, Amsterdam, North-Holland, 1979.

%\bibitem{LLBW} {R.S. Lakes, T. Lee, A. Bersie, and Y. Wang}, \emph{Extreme damping in composite materials
%with negative-stiffness inclusions}, Nature, {\bf 410} (2001), 565--567.
%
\bibitem{LiLiu2d} {H. Li and H. Liu}, \emph{On anomalous localized resonance for the elastostatic system}, SIAM J. Math. Anal., {\bf 48} (2016), 3322--3344.

\bibitem{LiLiu3d} {H. Li and H. Liu}, \emph{On three-dimensional plasmon resonance in elastostatics},
Annali di Matematica Pura ed Applicata, doi:10.1007/s10231-016-0609-0.

\bibitem{s25} {H. Li and H. Liu}, {\it On anomalous localized resonance and plasmonic cloaking beyond the quasistatic limit}, {Proceedings of the Royal Society A}, 474: 20180165.

\bibitem{LLL1} {H. Li, J. Li and H. Liu}, {On novel elastic structures inducing polariton resonances with finite frequencies and cloaking due to anomalous localized resonance}, {Journal de Math\'ematiques Pures et Appliqu\'ees}, {\bf 120}(2018), pp 195--219.

\bibitem{LLL} {H. Li, J. Li and H. Liu}, \emph{On quasi-static cloaking due to anomalous localized resonance in $\mathbb{R}^3$}, SIAM J. Appl. Math., {\bf 75}  (2015), no. 3, 1245--1260.

\bibitem{LLLW} {H. Li, S. Li, H. Liu and X. Wang}, {Analysis of electromagnetic scattering from plasmonic inclusions at optical frequencies and applications},  ESAIM: Math. Model. Numer. Anal. , arXiv:1804.09517.

%\bibitem{LASS} {A. E. H. Love}, \emph{A Treatise on the Mathematical Theory of Elasticity}, 4th Edition, Cambridge University Press, 2013.
%
\bibitem{GWM1} {R.C. McPhedran, N.-A.P. Nicorovici, L.C. Botten and G.W. Milton}, \emph{Cloaking by plasmonic
resonance among systems of particles: cooperation or combat?} C.R. Phys., {\bf 10} (2009), 391--399.

%\bibitem{GWM2} {D. A. B. Miller}, \emph{On perfect cloaking}, Opt. Express, {\bf 14} (2006), 12457--12466.
%
\bibitem{GWM3} {G.W. Milton and N.-A.P. Nicorovici}, \emph{On the cloaking effects associated with anomalous
localized resonance}, Proc. R. Soc. A, {\bf 462} (2006), 3027--3059.

\bibitem{GWM4} {G.W. Milton, N.-A.P. Nicorovici, R.C. McPhedran, K. Cherednichenko and Z. Jacob},
\emph{Solutions in folded geometries, and associated cloaking due to anomalous resonance}, New. J. Phys., {\bf 10} (2008), 115021.

\bibitem{GWM5} {G.W. Milton, N.-A.P. Nicorovici, R.C. McPhedran, and V.A. Podolskiy}, \emph{ Proof of
superlensing in the quasistatic regime, and limitations of superlenses in this regime due
to anomalous localized resonance}, Proc. R. Soc. A, {\bf 461} (2005), 3999--4034.

\bibitem{Jcn}
{J.~C.~N\'ed\'elec}, {\it Acoustic and Electromagnetic Equations: Integral Representations for Harmonic Problems},
Springer-Verlag, New York, 2001.

\bibitem{GWM6} {N.-A.P. Nicorovici, R.C. McPhedran, S. Enoch and G. Tayeb}, \emph{Finite wavelength cloaking
by plasmonic resonance}, New. J. Phys., {\bf 10} (2008), 115020.

\bibitem{GWM7} {N.-A.P. Nicorovici, R.C. McPhedran and G.W. Milton}, \emph{Optical and dielectric properties
of partially resonant composites}, Phys. Rev. B, {\bf 49} (1994), 8479--8482.

\bibitem{GWM8} {N.-A.P. Nicorovici, G.W. Milton, R.C. McPhedran and L.C. Botten}, \emph{Quasistatic cloaking
of two-dimensional polarizable discrete systems by anomalous resonance}, Optics
Express, {\bf 15} (2007), 6314--6323.
%




\end{thebibliography}
\end{document}